\documentclass[11pt,reqno]{article}
\usepackage[utf8]{inputenc}
\usepackage[a4paper, margin=1.5cm,top=3cm, bottom=3cm]{geometry}
\usepackage{amsmath, amsthm, amssymb, amsfonts}
\usepackage{mathtools}
\usepackage{todo}
\usepackage{hyperref}
\usepackage{enumitem}
\usepackage{setspace}
\usepackage{xcolor}
\usepackage{thmtools}
\usepackage{thm-restate}
\usepackage{verbatim}
\usepackage{enumitem}
\usepackage{enumitem} % for custom labels in lists
\usepackage{hyperref} % for clickable references
\usepackage{cleveref} % for smarter referencing (optional but nice)
\usepackage{tikz}
\usetikzlibrary{arrows.meta}
\usepackage{graphicx}
\usepackage{caption}
\usepackage{subcaption}
\usepackage[numbers,sort]{natbib} 

\newtheorem{theorem}{Theorem}[section]
\newtheorem{lemma}[theorem]{Lemma}
\newtheorem{claim}[theorem]{Claim}
\newtheorem{definition}[theorem]{Definition}
\newtheorem{corollary}[theorem]{Corollary}

\newtheorem{fact}[theorem]{Fact}

\newenvironment{claimproof}[1][\proofname]{\proof[#1]}{\endproof}

\title{Towards the  Lov\'{a}sz conjecture via sublinear expanders}
\author{Matija Buci{\'c}\thanks{Faculty of Mathematics, University of Vienna, Austria, and Department of Mathematics, Princeton University, USA. Email: \texttt{matija.bucic@univie.ac.at}. Research supported in part by NSF Award DMS-2349013.} \and Micha Christoph\thanks{Department of Mathematics, ETH Z\"{u}rich. Email: \texttt{micha.christoph@math.ethz.ch}. Research supported by SNSF Ambizione Grant No. 216071. } \and Alexey Pokrovskiy\thanks{Department of Mathematics, University College London, UK. Email: \texttt{a.pokrovskiy@ucl.ac.uk}.} \and Raphael Steiner\thanks{Department of Mathematics, ETH Z\"{u}rich. Email: \texttt{raphaelmario.steiner@math.ethz.ch}. Research supported by SNSF Ambizione Grant No. 216071.}}

\begin{document}
\maketitle
\begin{abstract}
Lov\'{a}sz' famous Hamiltonicity conjecture (1969) states that every connected vertex-transitive graph has a Hamiltonian path. A stronger version of the conjecture, often attributed to Thomassen (1978), states that every sufficiently large such graph even has a Hamiltonian cycle. Despite the great amount of attention these conjectures have attracted over the past decades both in the combinatorial and algebraic communities, for more than 40 years the best known lower bound for the maximum length of a cycle (path) in a connected vertex-transitive graph of order $n$ remained of the form $\Omega(\sqrt{n})$, due to Babai (1979). A series of recent works has successively improved the exponent in this lower bound further. In this paper, improving the previous state-of-the-art bound $\Omega(n^{9/14})$ due to Norin et al.~(2025), we prove that every connected vertex-transitive graph of order $n$ contains a cycle of length at least $n^{2/3-o(1)}$. This hits a natural barrier for several existing approaches from previous work.

Our proofs combine recent embedding techniques for paths in sublinear expanders, sublinear expander decompositions of almost-regular graphs, and several additional combinatorial ideas.
\end{abstract}
\section{Introduction}
\emph{Hamiltonicity} is one of the oldest and most fundamental notions in graph theory and has been studied since the beginnings of the subject.  A \emph{Hamilton} cycle (path) in a graph is a cycle (path) which covers all vertices, and a graph is called \emph{Hamiltonian} or \emph{traceable} if it contains a Hamilton cycle (path).  Apart from being a very natural object of study in combinatorics, Hamiltonicity also has strong ties to more applied algorithmic problems such as the \emph{Traveling Salesman Problem} in theoretical computer science. Deciding whether a graph has a Hamiltonian cycle (path) is one of the well-known 21 \textsc{NP}-hard problems from Karp's list of such problems~\cite{karp}, and thus easily checkable sufficient and necessary conditions for Hamiltonicity are of high interest and have been extensively studied. Probably the most famous such condition is Dirac's theorem from 1952, showing that a minimum degree bound of $n/2$ in an $n$-vertex graph is sufficient (and in general best possible) to guarantee a Hamilton cycle in a graph. While many generalizations of Dirac's condition have been discovered, there remains a surprising scarcity of sufficient conditions for Hamiltonicity that apply to very \emph{sparse} graphs of constant degree, say, and that in particular include the cycle graphs $C_n$ for $n\ge 3$.
One such potential condition which stands out in this sense was formulated conjecturally by Lov\'{a}sz in 1969~\cite{lovasz1969}. His conjecture, nowadays famous as the \emph{Lov\'{a}sz conjecture}, states that every connected vertex-transitive graph admits a Hamiltonian path. Vertex-transitive graphs (also known as symmetric graphs) can have a rich variety of structure and densities, and include large complicated classes of cubic graphs. Thus this condition, if true, would be quite orthogonal to the known sufficient conditions for Hamiltonicity that only apply to very dense graphs, such as Dirac's theorem. A well-known variant of Lov\'{a}sz' conjecture, often attributed to Thomassen (see~\cite{babai}) states that with finitely many exceptions, all connected vertex-transitive graphs even admit a Hamilton \emph{cycle}. Variants of Lov\'{a}sz' and Thomassen's conjectures, such as their restriction to the important special case of \emph{Cayley graphs} of finite groups, trace back even further, for instance to the work of Rankin~\cite{rankin} in 1948 and the work of Rapaport-Strasser in 1959~\cite{RS}. We remark in passing that also the corresponding questions for directed paths and cycles in vertex-transitive \emph{digraphs} are equally interesting, see~\cite{bucic} for a summary of the history of and some recent progress on these questions.

Lov\'{a}sz' and Thomassen's conjectures have attracted widespread attention in the last $50$ years, with many partial results on special cases, such as graphs of a particular order or Cayley graphs of specific groups, see~\cite{alspach81,kutnar,cayley-survey-84} for three surveys on these conjectures and~\cite{torsten,bedert,draganic,bradac} for some recent positive results confirming the conjecture for Kneser graphs and for moderately dense and random Cayley graphs, respectively. 

While these results all only apply to specific types of vertex-transitive graphs, such as those whose order is a particular type of number or those that correspond to specific Cayley graphs, \emph{general} positive results on the conjectures which apply to \emph{all} vertex-transitive graphs have remained few. Probably the most general such line of attack has been to give lower bounds on the function $\ell(n)$, defined as the maximum integer such that every connected vertex-transitive graph of order $n\ge 3$ contains a cycle of length at least $\ell(n)$. While the aforementioned stronger form of Lov\'{a}sz' conjecture claims that $\ell(n)=n$ for sufficiently large $n$, for over 40 years the best known lower bound remained stuck at $\ell(n)\ge\Omega(\sqrt{n})$ due to Babai~\cite{babai}. But recently, the problem has regained momentum with some exciting new advances. Namely, in a very nice paper from 2023, DeVos~\cite{devos} improved Babai's bound to $\ell(n)\ge \Omega(n^{3/5})$. This was further improved to $\ell(n)\ge \Omega(n^{13/21})$ by Groenland et al.~\cite{groen} and then to $\ell(n)\ge \Omega(n^{9/14})$ by Norin et al.~\cite{norin}, which has remained the state of the art thus far.
\paragraph{\textbf{Our results.}} In this paper, we improve the lower bound towards Lov\'{a}sz' and Thomassen's conjectures further to $\ell(n)\ge n^{2/3-o(1)}$, as formally stated in the following theorem.

\begin{theorem}\label{thm:mainlovasz}
Let $\varepsilon>0$. Then, for every sufficiently large number $n$, every connected vertex-transitive graph of order $n$ contains a cycle of length at least $n^{2/3-\varepsilon}$.
\end{theorem}

Importantly, Theorem~\ref{thm:mainlovasz} represents more than just another small nudge towards the conjectures\textbf{---}it actually achieves the maximum that was theoretically possible with the current approaches from previous work towards Lov\'{a}sz' and Thomassen's conjectures, and thus hits a natural barrier of the existing methods. To make this clear, in the next two paragraphs we summarize the two existing approaches towards improved lower bounds for $\ell(n)$ from the previous works~\cite{devos,groen,ma2025intersections,norin} which are all based on structural results about longest paths and cycles in general graphs, explain how our approach relates to these frameworks and how it exhausts the natural limits of both approaches. Along the way, we also state an additional result on longest path transversals implied by our proof.

\paragraph*{\textbf{Intersections of and hitting sets for longest paths and cycles.}} There are many fundamental open problems about basic properties of longest paths and cycles in graphs that remain wide open. The probably earliest question in this direction was raised already in 1966 by Gallai. To state the question formally, let us define a \emph{longest path (cycle) transversal} in a graph $G$ as a subset of vertices intersecting every longest path (cycle) in the graph, and let us denote by $\mathrm{lpt}(G), \mathrm{lct}(G)$, respectively, the minimum size of such a longest path (cycle) transversal in $G$. Motivated by the well-known folklore fact that any two longest paths in a connected graph must intersect, Gallai~\cite{gallai} asked whether every connected graph $G$ satisfies $\mathrm{lpt}(G)=1$. While this was shown to be false by Walther~\cite{walther}, the relaxed question~\cite{Zam76,Wal78} whether one can bound $\mathrm{lpt}(G)$ by an absolute constant for all connected graphs (and analogously, whether one can bound $\mathrm{lct}(G)$ by an absolute constant for all 
$2$-connected graphs) has remained open for 50 years. The best known upper bounds on these parameters in terms of the order of the graph are due to Norin et al.~\cite{norin}, who improved several previous bounds~\cite{Rau14,Lon21,Kie23} by showing that every connected $n$-vertex graph $G$ satisfies $\mathrm{lpt}(G)\le \sqrt{8n}$ and that every $2$-connected $n$-vertex graph $G$ satisfies $\mathrm{lct}(G)\le \sqrt{8n}$. 

Another fundamental question, which is attributed to Smith and can be traced back at least to 1980 (cf.~\cite{Groetschel84,Bondy95}), is the conjecture that in every $r$-connected graph with $r\ge 2$, every pair of longest cycles must intersect in at least $r$ vertices. The best known asymptotic lower bound on the intersection of two longest cycles in an $r$-connected graph was recently improved to $\Omega(r^{2/3})$ by Ma and Zhao~\cite{ma2025intersections}, improving another sequence of prior results~\cite{groen,Chen}.

Both of the above problems, i.e.\ upper bounds for longest path and cycle transversals as well as lower bounds for intersections of longest cycles, have direct applications to improving bounds for Lov\'{a}sz' and Thomassen's conjectures. To explain this, it is important to note that it was shown in~\cite[Proposition 5.1]{groen} that every connected vertex-transitive graph $G$ of order $n$ contains a cycle of length at least $\frac{n}{\mathrm{lct}(G)}$ and a path of length at least $\frac{n}{\mathrm{lpt}(G)}-1$. Hence, any progress on upper bounds for longest path and cycle transversals directly yields lower bounds for longest paths and cycles in connected vertex-transitive graphs. In fact, the aforementioned upper bounds of $\sqrt{8n}$ on longest path and cycle transversals due to Norin et al.~directly recover the well-known result of Babai~\cite{babai} that $\ell(n)\ge \Omega(\sqrt{n})$. However, as mentioned before, Norin et al.~in fact obtained a better lower bound of $\ell(n)\ge \Omega(n^{9/14})$. To do so, they derived another bound on longest path transversals in terms of the maximum length $\ell$ of a path rather than the order of the graph: They proved that $\mathrm{lpt}(G)\le O(\ell^{5/9})$ for every connected graph $G$ with maximum path length~$\ell$. Using the above-mentioned relationship that $\ell\ge \frac{n}{\mathrm{lpt}(G)}-1$ for connected vertex-transitive graphs, and some additional arguments (which we will repeat in the proof of Theorem~\ref{thm:mainlovasz} at the end of this section), one obtains the desired lower bound of $\ell(n)\ge \Omega(n^{9/14})$ from this approach. Norin et al. (see the discussion in the conclusion of~\cite{norin}) point out that the maximum limit achievable using methods similar to theirs would be to unify both of their bounds of the form $\mathrm{lpt}(G)\le O(\sqrt{n})$ and $\mathrm{lpt}(G)\le O(\ell^{5/9})$ to a bound of the form $\mathrm{lpt}(G)\le O(\sqrt{\ell})$, which with the same approach would yield a lower bound of $\ell(n)\ge \Omega(n^{2/3})$ towards Lov\'{a}sz' and Thomassen's conjectures. In this paper, we prove an asymptotic version of this goal. 

\begin{theorem}\label{thm:smallhittingset}
Let $G$ be a connected graph on at least two vertices and let $\ell$ denote the maximum length of a path in $G$. Then, $\mathrm{lpt}(G)\le \ell^{1/2+o(1)}$. 
\end{theorem}

Theorem~\ref{thm:mainlovasz} is a direct consequence of this result, and we would like to demonstrate the quick deduction before moving on to describing the second approach for lower bounding $\ell(n)$ related to Smith's conjecture.

\begin{proof}[Proof of Theorem~\ref{thm:mainlovasz}, assuming Theorem~\ref{thm:smallhittingset}]
Let $\varepsilon>0$ and let $n\in \mathbb{N}$ be sufficiently large as a function of $\varepsilon$. Let $G$ be a connected vertex-transitive graph of order $n$, and let $\ell$ denote the maximum length of a path in $G$. Then, by Theorem~\ref{thm:smallhittingset} we have $\mathrm{lpt}(G)\le \ell^{1/2+o(1)}$. Furthermore, by \cite[Proposition~5.1]{groen} we have $\ell\ge \frac{n}{\mathrm{lpt}(G)}-1$. It follows that
$$n\le (\ell+1)\mathrm{lpt}(G)\le \ell^{3/2+o(1)},$$ and hence $\ell\ge n^{2/3-o(1)}\ge \frac{5}{2}n^{2/3-\varepsilon}$, using that $n$ is chosen sufficiently large. 

Now, let $\ell'$ denote the maximum length of a cycle in $G$ and let $d\ge 2$ denote the degree of some (and thus every) vertex in $G$. By a result of Watkins~\cite{watkins}, we have that $G$ is $(\lfloor 2d/3\rfloor+1)$-connected. Hence, either $d=2$ in which case $G$ must be isomorphic to $C_n$ and hence satisfies $\ell'=n\ge n^{2/3-\varepsilon}$, or we have that $G$ is $3$-connected. In the latter case, a result of Bondy and Locke~\cite{bondy} guarantees that $\ell'\ge \frac{2}{5}\ell\ge n^{2/3-\varepsilon}$. Hence, in each case we obtain that $G$ contains a cycle of length at least $n^{2/3-\varepsilon}$, as claimed in the statement of the theorem. This concludes the proof.
\end{proof}

For Smith's conjecture, the connection to bounds on $\ell(n)$ is more subtle. In fact, a direct relation only exists to a natural ``local'' variant of Smith's problem: For every $k\in \mathbb{N}$, let us denote by $s(k)$ the smallest integer such that for every pair $C_1, C_2$ of longest cycles in a graph intersecting in at most $k$ vertices, there exists a vertex-separator of size at most $s(k)$ in $G$ that separates $C_1$ from $C_2$. Clearly, if $s(k)< r$ for some values of $k, r\in \mathbb{N}$ then any two longest cycles in an $r$-connected graph intersect in more than $k$ vertices. In fact, all the progress towards Smith's conjecture thus far has worked in this local setting and has proved corresponding upper bounds for $s(k)$, culminating in the aforementioned state of the art bound due to Ma and Zhao~\cite{ma2025intersections} who obtained $s(k)=O(k^{3/2})$. The trivial lower bound for $s(k)$ is $s(k)\ge k$, and it seems possible that $s(k)$ grows linearly in $k$. 

The connection between this ``local'' variant of Smith's conjecture and lower bounds on $\ell(n)$ is as follows: Groenland et al.~\cite{groen} implicitly proved that any upper bound of the form $s(k)=O(k^\alpha)$ for some $\alpha\in [1,\infty)$ immediately yields a lower bound of $\ell(n)\ge \Omega(n^{(1+\alpha)/(1+2\alpha)})$ for long cycles in connected vertex-transitive graphs. Hence, the hard limit for a lower bound on $\ell(n)$ that could potentially be derived from this approach is $\ell(n)\ge \Omega(n^{2/3})$. Our Theorem~\ref{thm:mainlovasz} asymptotically achieves this theoretical limit via a different approach.

\paragraph*{\textbf{Key technical result and overview.}} 

To prove Theorem~\ref{thm:smallhittingset} and thus our main result Theorem~\ref{thm:mainlovasz}, we improve a key lemma from the previous work of Norin et al.~\cite{norin} about finding a path using many edges of a matching spanned within a cycle. Concretely, we prove the following asymptotically tight result.

\begin{theorem}\label{thm: main}
    For every fixed $\varepsilon>0$ and every sufficiently large $n\in\mathbb{N}$, the following holds. Let $L$ be a cycle with cyclic vertex-sequence $v_1,\ldots,v_{2n}$ and $M$ a perfect matching between $v_1,\ldots,v_n$ and $v_{n+1},\ldots,v_{2n}$. Then, there exists a path in the graph $L\cup M$ which uses at least $n^{1-\varepsilon}$ edges of $M$.
\end{theorem}

Norin et al.~\cite[Lemma 8]{norin} proved an analogous result, but only for the fixed value $\varepsilon=0.2$. The main novelty of this paper lies in proving this key result (Theorem~\ref{thm: main}). The reduction to Theorem~\ref{thm: main} has already been observed in the work of Norin et al. Nevertheless, for the sake of completeness of the presentation and since there are a few formal differences, we decided to include the complete deduction of Theorem~\ref{thm:smallhittingset} from Theorem~\ref{thm: main} in the appendix of this paper. 

Our proof of Theorem~\ref{thm: main} combines recent embedding techniques for paths in sublinear expanders as well as the tool of expander decompositions of almost-regular graphs from the recent work of Letzter et al.~\cite{expanders} with an array of new combinatorial ideas to glue many short path segments into one big path which uses many edges of the matching $M$. In Section~\ref{sec:aux} we prepare the proof of our main result, Theorem~\ref{thm: main}, with several auxiliary results about embedding vertex-disjoint paths into sublinear expanders. In Section~\ref{sec:keylemma} we then prove a key lemma about cycle decompositions of Eulerian graphs, which will later be instrumental for proving Theorem~\ref{thm: main} (which is done in Section~\ref{sec:thm}): Here, given as input a graph $L\cup M$ as in the statement of Theorem~\ref{thm: main}, we first contract some segments of the cycle $L$ into single vertices to increase the average degree of the graph at hand to be polynomial in $n$. We then apply the tools from Section~\ref{sec:aux} to decompose this auxiliary graph into disjoint sublinear expander pieces, within which we can find suitable collections of disjoint paths, exhausting a significant amount of vertices of the expander pieces. Finally, we use the technical lemma from Section~\ref{sec:keylemma} to glue these paths from different pieces together into one path in the auxiliary graph such that it expands to a path in the original graph $L\cup M$ using many edges of $M$. This overview is extremely oversimplified, but we still think it is helpful to keep it in mind as a guide when reading through the following sections.
\paragraph*{\textbf{Notation and Terminology.}} 
Throughout, we denote by $\log x$ the \emph{natural} logarithm of a real number $x>0$. By $[k]:=\{1,\ldots,k\}$, we denote the set of the first $k$ positive integers. With a slight abuse of notation, given two positive real constants $\varepsilon_1, \varepsilon_2\in (0,1)$, we write $\varepsilon_1 \ll \varepsilon_2$ to indicate that for some (implicit but fixed) function $f:(0,1)\rightarrow (0,1)$, we assume that $\varepsilon_1<f(\varepsilon_2)$.

Given a graph $G$, we denote by $V(G)$ and $E(G)$ its vertex- and edge-set, respectively. For a subset of vertices $X$, we denote by $G[X]$ and $G-X$ the subgraph induced by $X$ and obtained by deleting $X$, respectively, and we use analogous notation for deletion of edge sets. We also denote by $N_G(X):=\{u\in V(G)\setminus X|\exists v\in X: uv\in E(G)\}$ the \emph{external neighborhood} of the set $X$ and simply write $N_G(v)=N_G(\{v\})$ for the neighborhood of a single vertex, as usual. 
We use the notation $\delta(G),\Delta(G),d(G)$ to denote the minimum, maximum and average degree of a graph $G$, respectively.
For two disjoint subsets $A,B\subseteq V(G)$, we denote by $e_G(A,B)$ the number of edges in $G$ with one end in $A$ and one end in $B$.

Given a finite set $V$ and a probability value $q\in [0,1]$, we will use the terminology \emph{$q$-random subset} to refer to a random subset of $V$ obtained by including each element with probability $q$, independently.

All graphs in this paper, unless explicitly referred to as multigraphs, are considered simple. In multigraphs, we do allow loops and multiple parallel edges between a pair of vertices. The degree of a vertex is then counted with multiplicities; in particular, a loop at a vertex contributes twice to the degree count at that vertex. In particular, this means that a multigraph of maximum degree at most $2$ consists of a disjoint union of paths and cycles, where cycles are allowed to have length two (two parallel edges) or one (a loop).

\section{Auxiliary results}\label{sec:aux}
In this section, we prepare the proof of our main technical result, namely Theorem~\ref{thm: main}, with several auxiliary results mostly on sublinear expanders. Most of these results will be deduced from prior work on sublinear expanders, in particular~\cite{expanders}. We start with the following central definition, which will be used repeatedly.
\begin{definition}
    Let $\beta,c,s>0$. A graph $H$ is called a \emph{$(\beta,c,s)$-expander} if, for every $U\subseteq V(H)$ and $F\subseteq E(H)$ with $1\leq|U|\leq \frac{2}{3}|V(H)|$ and $|F|\leq s|U|$, we have
    \begin{align*}
        |N_{H-F}(U)|\geq \frac{\beta}{(\log |V(H)|)^c}\cdot |U|.
    \end{align*}
\end{definition}
Our first lemma is a simple fact that can be straightforwardly deduced from the previous definition.
\begin{lemma}\label{lem: can remove}
    Let $0<\alpha<\varepsilon<1$ and $c>0$. Then, the following holds for every large enough $n\in \mathbb{N}$. Let $G$ be a $(1/8,c,s)$-expander with at most $n$ vertices satisfying $\delta(G)\geq n^{\varepsilon}/3$ and $s\le n^\varepsilon /100$. Let $X$ be an arbitrary set of at most $n^{\alpha}$ vertices. Then, $G-X$ is a $(1/16,c,s)$-expander.
\end{lemma}
\begin{proof}
    Let $U\subseteq V(G)\setminus X$ and $F\subseteq E(G-X)$ be arbitrary sets with $1\leq |U|\leq \frac{2}{3}|V(G-X)|$ and $|F|\leq s|U|$. Suppose first that $|U|\leq n^{\varepsilon}/6$. By the minimum degree condition, it follows that $e_G(U,V(G)\setminus U)\geq n^{\varepsilon}|U|/6$. Therefore, $e_{G-F}(U,V(G)\setminus U)\geq n^{\varepsilon}|U|/7$. But every vertex has at most $|U|$ neighbors in $U$, and it follows that $|N_{G-F}(U)|\geq \frac{e_{G-F}(U,V(G)\setminus U)}{|U|}\ge  n^{\varepsilon}/7$. Hence, $|N_{(G-X)-F}(U)|=|N_{G-F}(U)\setminus X|\geq n^{\varepsilon}/8\ge \frac{1}{16(\log |V(G-X)|)^c}|U|$, as desired. Next, suppose that $|U|>n^{\varepsilon}/6$. Then, using $|X|\le n^{\alpha}$ and $\alpha<\varepsilon$, we obtain, for sufficiently large $n$, 
    $$|N_{(G-X)-F}(U)|=|N_{G-F}(U)\setminus X|\geq\frac{1}{8(\log |V(G)|)^c}\cdot |U|-n^{\alpha}\geq \frac{1}{16(\log |V(G-X)|)^c}\cdot |U|.$$ 
    This concludes the proof that $G-X$ is a $(1/16,c,s)$-expander.
\end{proof}
Next, we shall state one of the main technical ingredients from~\cite{expanders}: A lemma guaranteeing in any near-regular graph $G$ a collection of vertex-disjoint near-regular sublinear expanders which cover almost all the vertices of $G$.
\begin{lemma}[\cite{expanders}, Lemma 4.1 applied with $\alpha =1$]\label{lem: exapnder decomposition}
    Let $\varepsilon, C, n, d$ satisfy $d \ge 2 \log n$, $0 \le \varepsilon \le (\log n)^{-(C-1)}$, $C \ge 57$ and let $n$ be large enough. Let $c = \frac{C(C-1)}{C-29}$. 

	Suppose that $G$ is an $n$-vertex graph with $\Delta(G) \le d$ and $d(G) \ge d(1- \varepsilon)$. Then, there is a collection $\mathcal{H}$ of vertex-disjoint subgraphs of $G$ such that every $H \in \mathcal{H}$ is a $(\frac{1}{8}, c, s_H)$-expander satisfying $d(H) \ge d(1 - \varepsilon_H)$ and $\delta(H) \ge d(H)/2$, where $s_H \coloneqq \frac{d}{4 (\log |V(H)|)^c}$, and $\varepsilon_H \coloneqq (\log |V(H)|)^{-(C-29)}$. Moreover, $\sum_{H \in \mathcal{H}} |V(H)| \ge (1 - \frac{(\log \log \log n)^2}{\log \log n})n$. 
\end{lemma}
We also require the following lemma from~\cite{expanders} in order to link up a given set of vertices in a sublinear expander with disjoint paths through a random set. 
\begin{lemma}[\cite{expanders}, Lemma 5.1, applied with $\varepsilon=1/16$]\label{lem: expander paths}
Let $c > 0$, $0 < q < 1$, $s \ge \frac{2 (\log n)^{9c+21}}{q^{10}}$, and let $n$ be sufficiently large as a function of\footnote{Importantly, in the statement of this lemma we are allowing $q$ to depend on $n$. While this is not explicitly stated in~\cite{expanders}, we carefully checked their proofs and the assumption that $q$ is constant is not made in their statements or proofs.} $c$. Suppose that $G$ is an $n$-vertex $(1/16, c, s)$-expander, and let $V$ be a $q$-random subset of $V(G)$. Then, with probability at least $1 - n^{-1}$, the following holds:

For every choice of distinct vertices $x_1, \ldots, x_r, y_1, \ldots, y_r \in V(G)\setminus V$ satisfying $|N_G(X)| \ge \frac{100 (\log n)^{7c+19}}{q^6}|X|$ for every $X \subseteq \{x_1, \ldots, x_r, y_1, \ldots, y_r\}$, there exists a sequence of pairwise vertex-disjoint paths $P_1, \ldots, P_r$ in $G$ such that, for each $i \in [r]$, $P_i$ is a path from $x_i$ to $y_i$ with internal vertices in $V$.
\end{lemma}
We record the following immediate consequence of the previous lemma, which allows us to amplify the success probability of the desired outcome at the cost of slightly increasing $q$.
\begin{corollary}\label{cor: expander paths}
    Let $c > 0$, let $n$ be sufficiently large as a function of $c$, let $0 < q \leq 1/\lceil\log n\rceil$ and $s\ge \frac{2 (\log n)^{9c+21}}{q^{10}}$. Suppose that $G$ is an $n$-vertex $(1/16, c, s)$-expander, and let $V$ be a $(\lceil\log n\rceil\cdot q)$-random subset of $V(G)$. Then, with probability at least $1 - n^{-\log n}$, the following holds.

    For every choice of distinct vertices $x_1, \ldots, x_r, y_1, \ldots, y_r\in V(G)\setminus V$ satisfying $|N_G(X)| \ge \frac{100 (\log n)^{7c+19}}{q^6}|X|$ for every $X \subseteq \{x_1, \ldots, x_r, y_1, \ldots, y_r\}$, there is a sequence of vertex-disjoint paths $P_1, \ldots, P_r$ in $G$ such that, for each $i \in [r]$, $P_i$ is a path from $x_i$ to $y_i$ with internal vertices in $V$.
\end{corollary}
\begin{proof}
    Let $V_1,\ldots,V_{\lceil\log n\rceil}$ be independent $q$-random subsets of $V(G)$. Since $\bigcup V_i$ is a $(1-(1-q)^{\lceil\log n\rceil})$-random subset and $(1-(1-q)^{\lceil\log n\rceil})\le q \lceil\log n\rceil$, we can couple these subsets to $V$ in such a way that $\bigcup V_i\subseteq V$ holds in every outcome. The probability that at least one of $V_1,\ldots,V_{\lceil \log n\rceil}$ satisfies the conclusion of Lemma~\ref{lem: expander paths} is at least $1-n^{-\lceil\log n\rceil}\ge 1-n^{-\log n}$. Note that if any of the $V_i$ satisfies this conclusion then so does $V$. This establishes the claim of the corollary.
\end{proof}
The following technical lemma combines the previous results into a ready-to-use statement which we will apply directly in our proof of Theorem~\ref{thm: main}. We need some additional terminology: Given two vertices $u,v$ in a graph $G$, we define a \emph{$u$-$v$-arc} in $G$ as simply a $u$-$v$-path if $u\neq v$, and as a cycle of positive length through $u=v$ if $u$ and $v$ are identical. In both cases, we refer to the \emph{internal vertices} of an arc as all its vertices distinct from $u$ and $v$, and refer to $u,v$ as its \emph{endpoints}.
\begin{lemma}\label{lem: expander connections}
    Let $0<\alpha,\varepsilon,\gamma<1$ be constants such that $\alpha +\gamma\leq\varepsilon/2$. Let $c>0$, and let $m\le n$ be positive integers such that $n$ is sufficiently large. Let $s\ge 2 (\log n)^{9c+31}n^{10\gamma/7}$. Suppose that $G$ is an $m$-vertex $(1/8, c, s)$-expander with $\Delta(G)\leq n^{\varepsilon}$ and $\delta(G)\geq n^{\varepsilon}/3$. Let $V_*\subseteq V(G)$ be a fixed set of at most $n^{\alpha}$ vertices and $V_0$ an $n^{-\gamma}$-random subset of $V(G)$. Then, with probability at least $1-n^{-\varepsilon^2\log n/4}$ the following holds: 

    \noindent For every $\emptyset \neq W\subseteq V_*\cup V_0$, and every multigraph $R$ with vertex-set $W$, maximum degree at most $2$ and at least one edge, there exists a collection of internally vertex-disjoint arcs $\mathcal P$ in $G$ with the following properties.
    \begin{itemize}
        \item Every arc $P\in \mathcal P$ has all internal vertices in $V(G)\setminus W$ and its endpoints in $W$.
        \item Every cycle $C$ in $R$ has a vertex $v_C$ such that no vertex in $V(C)\setminus \{v_C\}$ is an endpoint of any arc in $\mathcal P$, and the graph $R'$ obtained by removing from every cycle $C$ in $R$ all edges and all vertices besides $v_C$ and adding an edge $uv$ if $\mathcal P$ contains an arc from $u$ to $v$ is a cycle (possibly a loop or two parallel edges). 
        \item The arcs in $\mathcal P$ contain a total of at least $m\cdot n^{-\gamma}/2$ edges.
    \end{itemize}
\end{lemma}
\begin{proof}
     First observe that $m=|V(G)|>\delta(G)\ge n^\varepsilon/3$ and hence $m$ is sufficiently large whenever $n$ is. Also, note that we have $2(\log n)^{9c+31}n^{10\gamma/7}\le n^{\varepsilon}/100$ since $\gamma\leq \varepsilon/2$ and $n$ is sufficiently large. This in particular implies that we may without loss of generality assume $s\le n^\varepsilon/100$ (since every $(1/8,c,s)$-expander is also a $(1/8,c,s')$-expander for every $s'\le s$).

     We may now apply Lemma~\ref{lem: can remove} to the vertex-set $V_*$ in $G$, and find that $G' = G-V_*$ is a $(1/16,c,s)$-expander with $\delta(G')\geq n^{\varepsilon}/3-n^{\alpha}\geq n^{\varepsilon}/4$. Also, $G'$ has $m'\ge m-n^\alpha\ge \frac{3}{4}m$ vertices.

     Let us randomly and independently color the vertices of $G'$ with colors from $\{0,1,2,3,4,5\}$, where we assign colors $0,4,5$ with probability $n^{-\gamma}$ each and colors $1,2,3$ with probability $n^{-\gamma/7}$ each, while some vertices remain uncolored. Denote by $V_i'$ the set of vertices with color $i$. Importantly, we ensure to couple this random coloring to $V_0$ in such a way that $V_{0}'=V_0\setminus V_*$.
     Since each $V_i'$ is distributed as a binomial random subset of $V(G')$, the probability that each of $V_1',V_2',V_3'$ satisfies the conclusion of Corollary~\ref{cor: expander paths} with $q=\frac{1}{n^{\gamma/7}\lceil \log m'\rceil}$ is at least $1-3m'^{-\log m'}\geq 1-n^{-\varepsilon^2\log n/2}$. Furthermore, a standard application of the Chernoff bound together with a union bound over all vertices shows that with probability at least $1-e^{-n^{\varepsilon-\gamma}/2}$ every vertex in $G$ has at most $6n^{\varepsilon-\gamma}$ neighbors in $V_0'\cup V_4'\cup V_5'$ (here we use the assumption $\Delta(G)\le n^\varepsilon$). Similarly, with probability at least $1-e^{-n^{\varepsilon-\gamma}/25}$, every vertex in $V_*$ has at least $n^{\varepsilon-\gamma}/6$ neighbors in $V_4'$ (using the assumption $\delta(G)\ge n^\varepsilon/3$). Finally, we also obtain that $|V_5'|\geq n^{-\gamma}m/2$ with probability at least $1-e^{-n^{-\gamma}m/24}\ge 1-e^{-n^{\varepsilon-\gamma}/72}$. All of these events happen simultaneously with probability at least $1-n^{-\varepsilon^2\log n/4}$. We will now show that deterministically, conditioned on these events, the event in the lemma statement holds. Hence, the latter also will hold with probability at least $1-n^{-\varepsilon^2\log n/4}$, which will conclude the proof.

     So let $W\subseteq V_*\cup V_0$ be an arbitrary subset and $R$ an arbitrary non-empty multigraph with vertex-set $W$ and maximum degree $2$. In particular, $R$ is a disjoint union of paths and cycles (where the latter may have length $1$ or $2$). For every cycle $C$ in $R$ let us pick an arbitrary vertex $v_C$ on that cycle and let $W'\subseteq W$ be the set obtained by keeping, from each cycle $C$, only the representative $v_C$ and by keeping all vertices not lying on cycles of $R$. Let $R_{W'}\subseteq R[W']$ be obtained by removing all loops. Then, $R_{W'}$ is a union of paths and isolated vertices. Let $H$ be an arbitrary multigraph on $W'$ such that $R_{W'}\cup H$ is a cycle spanning all vertices of $W'$. Note that $\Delta(H)\leq 2$ and $H$ contains at least one edge since $R_{W'}$ is not a cycle. 
     Additionally, $H$ contains at most $1$ loop and if it does then $|W'|=1$.

     Now, we show there exists a collection of internally vertex-disjoint arcs $\mathcal P$ such that there is exactly one arc with endpoints $u$ and $v$ in $\mathcal P$ for every edge $uv$ in $H$, while all the internal vertices of these arcs are part of $V_1'\cup V_2'\cup V_3'\cup V_4'\cup V_5'$ and hence disjoint from $W$. We further ensure that $\mathcal P$ contains at least $m\cdot n^{-\gamma}/2$ edges in total by choosing it such that every vertex in $V_5'$ appears as an internal vertex of some arc in $\mathcal{P}$. Note that such a collection of arcs $\mathcal P$ will certify that the event in the lemma statement holds, as desired.

     We cannot directly apply Corollary~\ref{cor: expander paths}, as some vertices of $H$, namely $V(H)\cap V_*$, are not in $G'$. To circumvent this issue, we make the following adjustment to $H$. For every $v\in V_*$, let us greedily pick two neighbors $v_1,v_2$ of $v$ in $V_4'$ such that all the chosen neighbors are distinct. Note that this is possible since $|V_*|\leq n^\alpha$ and each $v\in V_*$ has at least $n^{\varepsilon-\gamma}/6\geq 2n^\alpha$ neighbors in $V_4'\setminus V_*$. Also note that, since $V(H)\subseteq W'\subseteq V_\ast\cup V_0$, which is disjoint from $V_4'$, all these chosen neighbors are not contained in $H$. Let us now adjust $H$ to obtain a multigraph $H'$, by, in some order, for every $v\in V_*$ with incident edges $vu,vw$, removing $v$ from $H$, adding $v_1$ and $v_2$ and replacing the edges $vu,vw$ with $v_1u,v_2w$, where we add only the edge $v_1u$ if $v$ is incident to only one edge $vu$, we add no edges at all if $v$ is isolated and, in the special case that $v$ is incident to a loop in $H$, we just add the edge $v_1v_2$. Observe that $\Delta(H')\le 2$ still holds after this modification, since $V(H)\subseteq V_*\cup V_0$ by definition and hence all the new vertices we added do not coincide with old vertices of $H$. 
     Further, if $H'$ contains a loop then so does $H$. As we observed earlier, this implies that $|W'|=1$ and, by the definition of $H'$, $H=H'$. We get that if $H'$ contains a loop then this loop is the only edge in $H'$.

     We now aim to find a collection of internally vertex-disjoint arcs $\mathcal P'$ in $G$ with internal vertices in $V_1'\cup V_2' \cup V_3' \cup V_5'$ such that for every edge $uv$ in $H'$ there is an arc in $\mathcal P'$ from $u$ to $v$, while ensuring that every vertex in $V_5'$ occurs as the internal vertex of some arc in $\mathcal P'$. Note that such $\mathcal P'$ can be extended to the desired collection of arcs $\mathcal P$ in $G$ using the edges $vv_1$ and $vv_2$ where necessary.

     Before finding this collection $\mathcal{P}'$, let us make a last adjustment to $H'$ to encode the requirement that the collection of arcs contains all of $V_5'$. Let $e$ be an arbitrary edge of $H'$ and create a new graph $H''$, obtained from $H'$ by replacing $e$ with an arc whose internal vertices span $V_5'$. Note that since $V(H')\subseteq V_0'\cup V_4'$, after this modification we still have $\Delta(H'')\le 2$. 
      Further, $H''$ does not contain a loop, since if $H'$ has a loop then $e$ is this loop and $H''$ is then a cycle of length at least $|V_5'|\geq 2$ as $H'$ has no other edges. 
     If we can now find a collection $\mathcal P''$ of internally vertex-disjoint arcs in $G$ with internal vertices from $V_1'\cup V_2'\cup V_3'$ such that for every edge $uv$ in $H''$ there is an arc from $u$ to $v$ in $\mathcal P''$, then by concatenating the arcs in $\mathcal P''$ with matching endpoints in $V_5'$ we obtain a collection of arcs $\mathcal P'$ as described above and, hence, the desired collection of arcs $\mathcal P$.

     It remains to find $\mathcal{P}''$ with the desired properties. We aim to use the conclusion of Corollary~\ref{cor: expander paths} for $V_1',V_2',V_3'$, where our pairs $x_1,\ldots,x_r,y_1,\ldots,y_r$ are vertices from $V(H'')$. Let $X\subseteq V(H'')$ be arbitrary. Then, since $X\subseteq V_0'\cup V_4'\cup V_5'$, every vertex has at most $6n^{\varepsilon-\gamma}$ neighbors in $X$. Together with $\delta(G')\geq n^{\varepsilon}/4$, we conclude $|N_{G'}(X)|\geq n^{\gamma}|X|/24$, which satisfies the requirement of Corollary~\ref{cor: expander paths} with $q=\frac{1}{n^{\gamma/7}\log m'}$.

     Since $\Delta(H'')\leq 2$, there exists a partition of $H''$ into three matchings $M_1,M_2,M_3$. Now, for each matching $M_i$, by the properties guaranteed by Corollary~\ref{cor: expander paths}, there exists a collection of vertex-disjoint paths $\mathcal P_i$ such that the internal vertices are from $V_{i}'$ and for every edge $uv\in M_i$, there is a path in $\mathcal P_i$ from $u$ to $v$. Finally, let $\mathcal P''=\mathcal P_1\cup\mathcal P_2\cup\mathcal P_3$ and note that the paths remain internally vertex-disjoint, as the internal vertices of $\mathcal P_i$ are from $V_{i}'$. As explained before, from $\mathcal{P}''$ we can obtain $\mathcal{P}'$ which in turn yields a collection of arcs $\mathcal{P}$ in $G$ with the desired properties. This concludes the proof of the lemma.
\end{proof}
%%%%%%%%%%%%%%%%%%%%%%%%%%%%%%%%%%%%%%%%%%%%%%%%%%%%%%%%%%%%%

Lemma~\ref{lem: expander connections} provides us with a tool for finding paths in expanders given that most of the endpoints lie in an $n^{-\gamma}$-random subset of $V(G)$. However, in the proof of Theorem~\ref{thm: main}, the endpoints we want to connect will follow a different random distribution. The following lemma couples these two distributions together, so that we will be able to apply Lemma~\ref{lem: expander connections} nonetheless.
\begin{lemma}\label{lem: coupling}
    Let $R$ be a multigraph with maximum degree at most $2$. Consider selecting each edge of $R$ independently with probability $p\in \left[0,\frac{1}{4}\right]$, and define $V'\subseteq V(R)$ as the set of all vertices for which at least one incident edge was selected. 

    Then, there exists a $2\sqrt{p}$-random subset $V$ of $V(R)$ coupled to the random process described above in such a way that $V'\subseteq V$ holds in every outcome.
\end{lemma}
\begin{proof}
    Instead of selecting each edge $e=uv$ with probability $p$, sample $X_u^e$ and $X_v^e$ independently, Bernoulli with probability $\sqrt p$, and set $X_e = X_u^e\cdot X_v^e$ (if $e$ is a loop, use two independent Bernoulli variables corresponding to its two incidences at $u$). Then, using $X_e$ as the indicator for the edge $e$, we obtain a subgraph of $R$ containing each edge independently with probability $p$ from which we can select $V'$ as the set of vertices incident to at least one edge. Now define $V\subseteq V(R)$ as the set of all vertices $u$ for which $X^e_u=1$ for some edge $e$ incident to $u$. Then, the probability that $u$ is included in $V$ is independent of all other vertices and is at most $2\sqrt{p}$ as every vertex is incident to at most $2$ edges in $R$. To conclude, note that $V'\subseteq V$.
\end{proof}
\section{Key technical lemma}\label{sec:keylemma}
In this section, we prove a key technical lemma about cycle decompositions of Eulerian graphs. Recall that a \emph{cycle decomposition} $\mathcal{C}$ of a (multi)graph $G$ is a collection of cycles in $G$ containing every edge of $G$ exactly once. As is well-known, a (multi)graph $G$ admits a cycle decomposition if and only if all its vertices have even degree. In the following, we will call \emph{connected} multigraphs with this property \emph{Eulerian}. 

In the remainder of this paper, given a collection $\mathcal C$ of cycles in a graph and some $p\in [0,1]$, we write $\mathcal C_p$ to denote a random collection of cycles obtained by selecting every cycle of $\mathcal C$ with probability $p$ independently of all other cycles. With a slight abuse of notation, for a collection of cycles $\mathcal{C}$ in a graph $G$ we will sometimes also use $\mathcal{C}$ to refer to the \emph{subgraph} of $G$ formed by the union of all the cycles in $\mathcal{C}$. In particular, $V(\mathcal{C})$ will refer to the vertex-set of said subgraph, and $\Delta(\mathcal{C})$ to its maximum degree.

In the following, given a weight function $w:V(G)\rightarrow \mathbb{Z}_{\ge 0}$ on the vertices of a graph $G$ and a subset of vertices $X$, we will use $w(X):=\sum_{v\in X}w(v)$ to denote the total weight of $X$. Similarly, we will use the notation $w(H):=w(V(H))$ for the total weight of vertices in a subgraph $H$ of $G$.

The following is the aforementioned key technical lemma that we aim to prove in this section.
\begin{lemma}\label{lem: cycles}
    Let $G$ be a non-empty Eulerian multigraph equipped with a cycle decomposition $\mathcal{C}$. Let $k\in \mathbb{N}$, $\delta \in (0,1)$, $p\in [k^{-\delta},1]$ and let $w:V(G)\rightarrow \mathbb{Z}_{\ge 0}$ be a weight function on $V(G)$ with positive total weight. Then, there exists a subcollection $\mathcal C^*\subseteq \mathcal C$ of cycles with $\Delta(\mathcal C^*)\leq 6k$ such that with probability at least $w(G)^{-\delta}/2$, the graph $\mathcal C^*\cup\mathcal C_p$ has a connected component of total weight at least $w(G)^{1-\delta}/2$.
\end{lemma}
As a first step towards Lemma~\ref{lem: cycles}, it will be useful to prove the following auxiliary ``loaded induction''-type statement first.
\begin{lemma}\label{lem: root cycles}
    Let $\mathcal C$ be a collection of edge-disjoint cycles in some graph, let $w:V(\mathcal{C})\rightarrow \mathbb{Z}_{\ge 0}$ be a non-negative integer weight function on its vertices, and let $\mathcal R \subseteq \mathcal C$ be a subcollection of cycles such that $\mathcal R$ contains at least one cycle in every connected component of $\mathcal C$. Let $k\in \mathbb{N}$, $\delta \in (0,1)$ and $p\in [k^{-\delta},1]$. Then, there exists another subcollection $\mathcal H\subseteq \mathcal C$ meeting each of the following conditions:
    \begin{enumerate}[label=(C\arabic*), ref=(C\arabic*)]
    \item \label{cond:C1} $|\mathcal{H}\cap \mathcal{R}|\le k$,
    \item \label{cond:C2} each cycle in $\mathcal H\cap \mathcal R$ intersects at most $2k$ cycles in $\mathcal H$ in some vertex,
    \item \label{cond:C3} every cycle in $\mathcal H$ intersects at most $3k$ cycles in $\mathcal H$ in some vertex, and, 
    \item \label{cond:C4}\textcolor{white}{hack}

    \vspace{-1.4cm}
    \begin{align}\label{exp: cycle expectation}
        \mathbb E[w(V)]\geq w(\mathcal C)^{1-\delta},
    \end{align}
    where $V$ is the set of vertices contained in connected components of $\mathcal C_p\cup \mathcal H$ which contain at least one cycle of $\mathcal R$.
\end{enumerate}
\end{lemma}
\begin{proof}
     We proceed by induction on $|\mathcal C|$, where the base case ($|\mathcal C|=0$) is trivial. Now suppose we are given a non-empty collection $\mathcal{C}$ of edge-disjoint cycles in some graph and that we have proved the assertion of the lemma for all edge-disjoint cycle collections smaller than $\mathcal{C}$. Let $w,\mathcal{R},k,\delta,p$ be given as in the lemma statement. Our goal in the following is to carry out the induction step, i.e.\ to show that the assertion of the lemma holds for $\mathcal{C}$.

     Let $t=|\mathcal R|$. Let $\mathcal K_1,\ldots,\mathcal K_t$ be disjoint subsets of $\mathcal C$ obtained by iteratively selecting $\mathcal K_i\subseteq \mathcal C\setminus(\cup_{j<i}\mathcal K_j)$ to be a set of cycles lexicographically maximizing $(w(V(\mathcal K_i)\setminus \bigcup_{j<i}V(\mathcal{K}_j)),|\mathcal K_i|)$ under the constraints that $\mathcal K_i$ is connected and $\mathcal K_i$ contains exactly one cycle of $\mathcal R$. Let $R_1,\ldots, R_t$ be the cycles in $\mathcal R$ such that $R_i\in \mathcal K_i$. The following claim introduces a central observation about this partition, which will be used repeatedly.
     \begin{claim}\label{clm: disjointness}
     We have that $\mathcal{K}_1,\ldots,\mathcal{K}_t$ form a partition of $\mathcal C$. Furthermore,  $V(\mathcal K_i)\cap V(\mathcal K_j\setminus \{R_j\})=\emptyset$ for all $1\le i<j\le t$.
     \end{claim}
     \begin{claimproof}
For the first part, suppose towards a contradiction that $\mathcal C\setminus \bigcup_{i=1}^t \mathcal K_i\neq \emptyset$. Since by assumption in the lemma, every component of $\mathcal C$ contains a cycle in $\mathcal R$, and since furthermore every cycle in $\mathcal R$ belongs to exactly one of $\mathcal K_1,\ldots,\mathcal K_t$, it now follows that there must exist some $C\in \mathcal C\setminus \bigcup_{i=1}^t \mathcal K_i$ such that $V(C)\cap V(\mathcal K_i)\neq \emptyset$ for some $1\le i \le t$. Then, $\mathcal{K}_i\cup \{C\}$ is also connected, still contains precisely one cycle from $\mathcal R$, and is disjoint from $\bigcup_{k<i}\mathcal K_k$. Hence, we obtain a contradiction to the maximality of $|\mathcal K_i|$ in our choice. 

The proof of the second part of the claim is analogous:
 Suppose towards a contradiction that there exists $C\in \mathcal K_j\setminus\{R_j\}$ such that $V(C)\cap V(\mathcal K_i)\neq \emptyset$. Then, again, $\mathcal K_i\cup \{C\}$ is connected, contains precisely one cycle from $\mathcal{R}$, and is still disjoint from $\bigcup_{k<i}\mathcal{K}_k$. So, as before, we obtain a contradiction to our choice of $\mathcal{K}_i$.
     \end{claimproof}

    For all $i\in [t]$, let us set $\mathcal K_i^-:= \mathcal K_i\setminus \{R_i\}$ and let $\mathcal R_i\subseteq \mathcal K_i^-$ denote the set of all cycles in $\mathcal{K}_i^-$ which intersect $R_i$ in at least one vertex. Let further $\mathcal K_{\leq k}:=\bigcup_{i=1}^{\min\{k,t\}} \mathcal K_i$, $\mathcal K_{\leq k}^-:=\bigcup_{i=1}^{\min\{k,t\}}\mathcal K_i^-$ and let $\mathcal R_{\leq k}$ denote the set of all cycles in $\mathcal{K}_{\le k}^-$ which share some vertex with at least one of the cycles $\{R_1,\ldots,R_{\min\{k,t\}}\}$.

    We next plan to apply induction to $\mathcal K_{\leq k}^-$ and $\mathcal R_{\leq k}$ to obtain a new collection $\mathcal H_{\leq k}\subseteq \mathcal K_{\le k}^-$, as well as, for each $k+1\leq i\leq t$, to $\mathcal K_i^-$ and $\mathcal R_i$ to obtain a new collection $\mathcal H_i\subseteq \mathcal K_i^-$. 

    Note that by definition, we have $R_1 \in \mathcal C \setminus \bigcup_{i=1}^{t}\mathcal K_i^-$ and $R_i\in \mathcal C\setminus \mathcal K_i^-$ for every $i$, so $|\mathcal{K}_{\le k}^-|<|\mathcal C|$ and $|\mathcal{K}_i^-|<|\mathcal{C}|$ for every $k+1\le i \le t$. Hence, to ensure that these inductive applications are valid, it suffices to verify that for every $i\in [t]$, every connected component of $\mathcal{K}_i^-$ contains at least one cycle in $\mathcal R_i$. Indeed, since $\mathcal{R}_i\subseteq \mathcal{R}_{\le k}$ for every $i\in [\min\{k,t\}]$, this will then also imply that every connected component of $\mathcal{K}_{\le k}^-=\bigcup_{i=1}^{\min\{k,t\}}\mathcal{K}_i^-$ contains at least one cycle from $\mathcal{R}_{\le k}$. 

   So let $i\in [t]$ be given to us arbitrarily. Observe that $\mathcal K_i=\mathcal K_i^-\cup \{R_i\}$ is connected by definition. Hence, every connected component of $\mathcal K_i^-$ intersects $V(R_i)$ and, thus, contains a cycle of $\mathcal R_i$, as desired. 

   Thus, the above-mentioned inductive applications are justified and the arising collections $\mathcal{H}_{\le k}\subseteq \mathcal{K}_{\le k}^-$ and $\mathcal{H}_i\subseteq \mathcal{K}_i^-$ for $k+1\le i\le t$ are well-defined. Finally, we set\begin{align*}
        \mathcal H := \{R_1,\ldots,R_{\min\{k,t\}}\}\cup \mathcal H_{\leq k}\cup \bigcup_{i=k+1}^t\mathcal H_i.
    \end{align*}
     and claim that $\mathcal H$ satisfies the conditions \ref{cond:C1}--\ref{cond:C4} in the lemma. Note that $\mathcal{H}\cap \mathcal{R}=\{R_1,\ldots,R_{\min\{k,t\}}\}$ and so \ref{cond:C1} holds. Also, for every $i\in [\min\{k,t\}]$ we have that $\{R_i\}$ is vertex-disjoint from all cycles in $\mathcal H_{\leq k}\setminus \mathcal R_{\leq k}$ by the definition of $\mathcal R_{\leq k}$ and from $\bigcup_{i=k+1}^t\mathcal H_i$ by Claim~\ref{clm: disjointness}. But $\mathcal H_{\leq k}\cap \mathcal R_{\leq k}$ contains at most $k$ cycles by \ref{cond:C1}, so each $R_i$ intersects at most $\min\{k,t\}+k\le 2k$ cycles in $\mathcal{H}$. This verifies that \ref{cond:C2} holds for $\mathcal{H}$. We move on to \ref{cond:C3}. Consider any $C\in\mathcal H$, we want to prove that $C$ intersects at most $3k$ other cycles in $\mathcal{H}$. To do so, let us distinguish four cases. If $C\in\{R_1,\ldots,R_{\min\{k,t\}}\}$, then the statement follows from \ref{cond:C2}, which we just proved. If $C\in \mathcal{H}_{\le k}\cap \mathcal R_{\leq k}$, then $C$ intersects at most $2k$ cycles in $\mathcal H_{\leq k}$ by induction and $C$ does not intersect $\bigcup_{i=k+1}^t\mathcal H_i$ by Claim~\ref{clm: disjointness}. It follows that $C$ intersects at most $2k+\min\{k,t\}\le 3k$ cycles in $\mathcal H$. Next, suppose $C\in \mathcal H_{\leq k}\setminus\mathcal R_{\leq k}$. Then, $C$ intersects neither $\{R_1,\ldots,R_{\min\{k,t\}}\}$ (by definition of $\mathcal R_{\leq k}$) nor $\bigcup_{i=k+1}^t\mathcal H_i$ (by Claim~\ref{clm: disjointness}). Therefore, we are done by induction. Finally, suppose $C\in \mathcal H_i$ for some $k+1\leq i\leq t$. By Claim~\ref{clm: disjointness}, all cycles in $\mathcal{H}$ which $C$ intersects in some vertex must lie in $\mathcal H_i$, and by induction (using condition \ref{cond:C3}), there are at most $3k$ cycles in $\mathcal{H}_i$ which $C$ intersects. This completes the proof that $\mathcal{H}$ satisfies condition~\ref{cond:C3}.

    It is left to show that $\mathcal{H}$ also satisfies \ref{cond:C4}. 
    Recall that $V\subseteq V(\mathcal C)$ in the lemma statement was defined as the set of all vertices contained in connected components of $\mathcal C_p\cup \mathcal H$ which contain at least one cycle of $\mathcal{R}$. We then have to show that $\mathbb{E}[w(V)]\ge w(\mathcal C)^{1-\delta}$. 

    Using the definitions of $\mathcal{H},\mathcal{H}_{\le k},\mathcal{R}_{\le k}$ and $\mathcal{H}_i, \mathcal{R}_i$, we can observe that each of the following sets of vertices is contained in $V$:
    \begin{itemize}
        \item The connected components of $(\mathcal{K}_{\le k}^-)_p\cup \mathcal{H}_{\le k}$ which contain at least one cycle in $\mathcal{R}_{\le k}$ (note that these components intersect vertices of at least one $R_i$, $i \in [\min\{k,t\}]$, all of which we include in $\mathcal H$) form part of $V$.
        \item For each integer $i$ with $k+1\le i\le t$ and such that $R_i\in \mathcal{C}_p$ (which happens with probability $p$), we have that all connected components of $(\mathcal{K}_i^-)_p\cup \mathcal{H}_i$ which contain at least one cycle in $\mathcal{R}_i$ (and hence intersect $R_i\in \mathcal{R}$) are contained in $V$.
        \item We have $\bigcup_{i=1}^{\min\{k,t\}}V(R_i)\subseteq V$, and each vertex in $\bigcup_{i=k+1}^{t}V(R_i)$ is contained in $V$ with probability at least $p$. 
    \end{itemize}

Pause to note that by Claim~\ref{clm: disjointness}, the sets described in the first two bullets above are pairwise disjoint from one another. Furthermore, by inductive applications of~\ref{cond:C4}, we have that the expected total weight of vertices of the first type is at least $w(\mathcal{K}_{\le k}^-)^{1-\delta}$ and the expected total weight of the vertices of the second type is at least $p\sum_{i=k+1}^{t} w(\mathcal{K}_i^-)^{1-\delta}$. Hence, the expected total weight of vertices of the first and second types is at least
$$w(\mathcal{K}_{\le k}^-)^{1-\delta}+p\sum_{i=k+1}^{t} w(\mathcal{K}_i^-)^{1-\delta}.$$

Finally, observe that the additional weight provided by the third type of vertices described above is, in expectation, at least $w(\mathcal{K}_{\le k})-w(\mathcal{K}_{\le k}^-)+p\cdot w(V(\mathcal C)\setminus V\left(\mathcal{K}_{\le k}\cup \bigcup_{i=k+1}^t\mathcal{K}_i^-\right))$. 
All in all, we thus find that
\begin{align*}
        \mathbb E[w(V)]&\geq w(\mathcal{K}_{\le k}^-)^{1-\delta}+w(\mathcal{K}_{\le k})-w(\mathcal{K}_{\le k}^-)+p\left(\sum_{i=k+1}^{t} w(\mathcal{K}_i^-)^{1-\delta}\right)+p\cdot w\left(V(\mathcal C)\setminus V\left(\mathcal{K}_{\le k}\cup \bigcup_{i=k+1}^t\mathcal{K}_i^-\right)\right) \\
        &\ge w(\mathcal{K}_{\le k})^{1-\delta}+p\left(\left(\sum_{i=k+1}^{t} w(\mathcal{K}_i^-)^{1-\delta}\right)+w\left(V(\mathcal C)\setminus V\left(\mathcal{K}_{\le k}\cup \bigcup_{i=k+1}^t\mathcal{K}_i^-\right)\right)\right), 
    \end{align*}
    where we used the easily verified fact that $a^{1-\delta}+b\ge (a+b)^{1-\delta}$ for all $a,b\in\mathbb{Z}_{\ge 0}$ and $\delta \in (0,1)$.

    In the following, for every $i\in [t]$ let us denote $Y_i:=~V(\mathcal{K}_i)\setminus \bigcup_{j<i}V(\mathcal{K}_j)$. Note that by our definition of $\mathcal{K}_1,\ldots,\mathcal{K}_t$, we then have $w(Y_1)\ge w(Y_2)\ge \cdots\ge w(Y_t)$. Note in addition that the sets $Y_1,\ldots,Y_t$ form a partition of $V(\mathcal{C})$. We can then rewrite $V(\mathcal{K}_{\le k})=Y_1\cup\cdots\cup Y_{\min\{k,t\}}$. Further note that $\bigcup_{i=k+1}^t Y_i=V(\mathcal C)\setminus V(\mathcal{K}_{\le k})$ , and hence  $$V(\mathcal C)\setminus V\left(\mathcal{K}_{\le k}\cup \bigcup_{i=k+1}^t\mathcal{K}_i^-\right)=\bigcup_{i=k+1}^t Y_i\setminus \bigcup_{i=k+1}^{t}V(\mathcal{K}_i^-)=\bigcup_{i=k+1}^{t}(Y_i\setminus V(\mathcal{K}_i^-)),$$ where in the last step we used that $Y_i\cap V(\mathcal{K}_j^-)=\emptyset$ whenever $i\neq j\in [t]$, which follows from Claim~\ref{clm: disjointness} and the definition of $Y_1,\ldots,Y_t$. This allows us to deduce the following lower bound on $\mathbb{E}[w(V)]$:
\begin{align*}\mathbb{E}[w(V)]&\ge w(\mathcal{K}_{\le k})^{1-\delta}+p\left(\left(\sum_{i=k+1}^{t} w(\mathcal{K}_i^-)^{1-\delta}\right)+w\left(V(\mathcal C)\setminus V\left(\mathcal{K}_{\le k}\cup \bigcup_{i=k+1}^t\mathcal{K}_i^-\right)\right)\right)\\
    &=w(\mathcal{K}_{\le k})^{1-\delta}+p\left(\sum_{i=k+1}^{t} w(\mathcal{K}_i^-)^{1-\delta}+\sum_{i=k+1}^t w(Y_i\setminus V(\mathcal{K}_i^-))\right)\\
    &\ge w(\mathcal{K}_{\le k})^{1-\delta}+p\sum_{i=k+1}^t (w(\mathcal{K}_i^-)^{1-\delta}+w(Y_i\setminus V(\mathcal{K}_i^-)))\\
    &\ge w(\mathcal{K}_{\le k})^{1-\delta}+p\sum_{i=k+1}^t w(Y_i)^{1-\delta},
    \end{align*} 
    where we again used the inequality $a^{1-\delta}+b\ge (a+b)^{1-\delta}$ in the last step.

    Before further estimating the above expression, it will be useful to prove the following simple inequality for non-negative real numbers.

    \begin{fact}\label{fact: inequality}
    Let $q\in\mathbb{Z}_{\ge 0}$, $k \in \mathbb{N}$, $\delta\in(0,1)$, and $p\ge k^{-\delta}$ as well as $x, y_1,\ldots,y_q\geq 0$ be real numbers such that $y_i\le \frac{x}{k}$ for each $i\in [q]$. Then,
    $$
    x^{1-\delta}+p\sum_{i=1}^q y_i^{1-\delta}\geq \left(x+\sum_{i=1}^qy_i\right)^{1-\delta}.
    $$
\end{fact}
\begin{claimproof}
  The statement trivially holds when $x=0$, so moving on suppose $x>0$. We then have 
    \begin{align*}
        &x^{1-\delta}+p\sum_{i=1}^q y_i^{1-\delta}\geq x^{1-\delta}+p\sum_{i=1}^q y_i\cdot (x/k)^{-\delta}\\ &\ge \left(x+\sum_{i=1}^{q}y_i\right)\cdot x^{-\delta}\ge \left(x+\sum_{i=1}^qy_i\right)^{1-\delta}.
    \end{align*}    
    \vspace{-0.8cm}

\textcolor{white}{hack}
\end{claimproof}

Let now $x:=w(\mathcal{K}_{\le k})=w(Y_1)+\cdots+w(Y_{\min\{k,t\}})$, $q:=t-\min\{k,t\}$, and $(y_1,\ldots,y_{q}):=\linebreak(w(Y_{k+1}),\ldots,w(Y_t))$. Since $w(Y_i)\le w(Y_{\min\{k,t\}})\le \frac{1}{\min\{k,t\}}x$ for every $k+1\le i \le t$, we can see that the numbers $q,x,y_1,\ldots,y_q$ satisfy the conditions of Fact~\ref{fact: inequality} (also if $k>t$, since then $q=0$ and the conditions are trivially satisfied). Hence, we obtain
 \begin{align*}
        \mathbb{E}[w(V)]&\ge w(\mathcal K_{\leq k})^{1-\delta}+p\sum_{i=k+1}^t w(Y_i)^{1-\delta}=x^{1-\delta}+p\sum_{j=1}^q y_j^{1-\delta}\\
        &\ge \left(x+\sum_{j=1}^q y_j\right)^{1-\delta}=\left(\sum_{i=1}^tw(Y_i)\right)^{1-\delta}=w(\mathcal{C})^{1-\delta},
    \end{align*}
    as desired. This shows that $\mathcal{H}$ also satisfies condition~\ref{cond:C4} of the inductive assertion, concluding the proof of the lemma.
\end{proof}

Lemma~\ref{lem: cycles} can now be straightforwardly deduced from Lemma~\ref{lem: root cycles}.

\begin{proof}[Proof of Lemma~\ref{lem: cycles}]
    Since $w(\mathcal{C})=w(G)>0$ by assumption of the lemma, we have that $\mathcal{C}\neq \emptyset$.
    Pick some cycle $R\in\mathcal C$ arbitrarily, and set $\mathcal{R}:=\{R\}$. Since $G$ is Eulerian and hence connected, we have that $\mathcal{R}$ has at least one cycle in every connected component of $\mathcal{C}$. Hence, we may apply Lemma~\ref{lem: root cycles}, which yields a subcollection of cycles $\mathcal{H}\subseteq \mathcal C$ such that every cycle in $\mathcal{H}$ intersects at most $3k$ cycles in $\mathcal{H}$ in some vertex, and $\mathbb E[w(V)]\geq w(G)^{1-\delta}$, where $V$ is the set of vertices contained in the unique connected component of $\mathcal C_p\cup \mathcal H$ which contains $R$ (and $V:=\emptyset$ if such a component does not exist). Since every cycle has maximum degree two, it follows from the first property of $\mathcal{H}$ that $\Delta(\mathcal H)\leq 6k$. Furthermore, since $V$ in any outcome has total weight at most $w(G)$, it follows by applying Markov's inequality to the non-negative random variable $w(G)-w(V)$ that $$\Pr[w(V)\ge w(G)^{1-\delta}/2]=\Pr\left[w(G)-w(V)\le w(G)-w(G)^{1-\delta}/2\right]$$
    $$=1-\Pr\left[w(G)-w(V)> w(G)-w(G)^{1-\delta}/2\right]\ge 1-\frac{w(G)-\mathbb{E}[w(V)]}{w(G)-w(G)^{1-\delta}/2}$$
    $$=\frac{\mathbb{E}[w(V)]-w(G)^{1-\delta}/2}{w(G)-w(G)^{1-\delta}/2}\ge \frac{w(G)^{1-\delta}/2}{w(G)}=\frac{w(G)^{-\delta}}{2},$$ as desired. This concludes the proof of the lemma.
\end{proof}
\section{Proof of Theorem \ref{thm: main}}\label{sec:thm}
We are finally ready to present the proof of the main technical result of this paper, Theorem~\ref{thm: main}.
 \begin{proof}[Proof of Theorem~\ref{thm: main}]

For convenience, we will only show that provided $\varepsilon>0$ is sufficiently small and rational and $n$ is sufficiently large in terms of $\varepsilon$ such that $n^\varepsilon$ is an integer dividing $n$, we can guarantee a path in $L\cup M$ using at least $n^{1-3\varepsilon}-1$ edges of $M$. It is not difficult to check that this suffices to prove the assertion of the theorem.

Let us now start the proof under these assumptions on $\varepsilon$ and $n$. Suppose towards a contradiction that there is no path in $L\cup M$ using at least $n^{1-3\varepsilon}-1$ edges of $M$.

  Fix some $0<\alpha\ll\varepsilon$ bounded away from zero such that $n^\alpha$ is an integer (the precise conditions on $\alpha$ become clear during the proof). Let $\mathcal I$ be the set of intervals of vertices of the form $\{v_{(i-1)n^\varepsilon+1},\ldots,v_{in^\varepsilon}\}$, $1\leq i\leq 2n^{1-\varepsilon}$. Given a graph $G$ on $V(L)$, we write $G\slash\mathcal I$ for the multigraph obtained by contraction of the intervals (here, we keep loops as well as multiple parallel edges between pairs of vertices, one for each connecting edge of the original intervals). In the following, let $G_M=M\slash\mathcal I$. Note that $G_M$ is a bipartite multigraph with no loops which is $n^{\varepsilon}$-regular. Let $G_L$ be obtained from $L\slash\mathcal I$ by removing the loops. Note that $G_L$ is a cycle.
  \begin{claim}\label{clm:simple}
      There is a simple, spanning subgraph $G'_M$ of $G_M$ satisfying $d(G'_M)\geq n^{\varepsilon}-1$ and $\Delta(G_M')\le n^\varepsilon$.
  \end{claim}
  \begin{claimproof}
      Let $T\subseteq G_M$ be the simple graph connecting two intervals if and only if they are connected by at least $2$ edges in $G_M$. Suppose that $e(T)\geq n^{1-2\varepsilon}$. As the maximum degree of $T$ is at most $n^\varepsilon/2$, we have that $T$ contains a matching $N$ of size at least $\frac{e(T)}{n^{\varepsilon}}\geq n^{1-3\varepsilon}$. Consider the path obtained by altering the subpath $P=v_1,\ldots,v_n$ of $L$ in the following way. For every edge $II'\in N$, where $I$ is an interval on $P$ and $I'$ an interval on $v_{n+1},\ldots, v_{2n}$, we replace the interval $I$ on $P$ with a detour through $I'$ traversing exactly two edges connecting $I$ and $I'$ in $G_M$. Note that since $N\subseteq T$, there are indeed at least two such edges. Then, the final path uses $2|N|\geq n^{1-3\varepsilon}-1$ edges of $M$, contradicting the assumption that there is no such path. Therefore, $e(T)<n^{1-2\varepsilon}$. Let $G'_M\subseteq G_M$ be the simple graph obtained by removing all multiple edges. Clearly, $G_M'$ has maximum degree at most $n^\varepsilon$. Since $e(T)<n^{1-2\varepsilon}$, it follows that when passing from $G_M$ to $G_M'$ we removed at most $n^\varepsilon\cdot n^{1-2\varepsilon}$ edges. As $G_M$ is $n^\varepsilon$-regular on $2n^{1-\varepsilon}$ vertices, it follows that $d(G'_M)\geq n^\varepsilon-1$.
  \end{claimproof}
  We now apply Lemma~\ref{lem: exapnder decomposition} to $G_M'$ with parameters $\varepsilon':=(\log n)^{-56}$, $C':=57$, $n':=|V(G_M')|=2n^{1-\varepsilon}$, $d':=n^\varepsilon$, $c':=\frac{C'(C'-1)}{C'-29}=114$. Pause to note that by Claim~\ref{clm:simple} and since $n$ is sufficiently large, $G_M'$ with these parameters indeed meets the conditions of Lemma~\ref{lem: exapnder decomposition}.
  Let $\mathcal{G}$ be the collection of vertex-disjoint subgraphs of $G_M'$ obtained from the lemma. Then, the following hold:
  \begin{itemize}
      \item every graph $G\in\mathcal{G}$ is a $(\frac{1}{8},114,s_G)$-expander, where $s_G:=\frac{n^\varepsilon}{4(\log |V(G)|)^{114}}$,
      \item every graph $G\in\mathcal{G}$ satisfies $d(G)\ge (1-(\log |V(G)|)^{-28}) n^\varepsilon$ and $\delta(G)\ge d(G)/2$, and
      \item $Z:=V(G_M')\setminus \bigcup_{G\in\mathcal{G}}V(G)$ has size $o(n^{1-\varepsilon})$.
  \end{itemize}

As a short interlude, we introduce some definitions. The vertices of the graphs in $\mathcal G$ really correspond to \emph{intervals} of $L$, though sometimes we associate vertices of $L$ with vertices in $\mathcal G$. To do so, for $G\in\mathcal G$, we write $V_L(G)$ to denote the set of vertices of $V(L)$ which appear in an interval which is a vertex of $G$. A \textit{connecting-system} is a pair $(\mathcal S,F)$ where $\mathcal S\subseteq\mathcal G$ and where $F\subseteq L\cup M$ is a non-empty linear forest whose components are paths of length at least one, and such that the following hold:
\begin{itemize}
    \item $F$ is edge-disjoint from the edges inside any interval $I\in \bigcup_{G\in\mathcal{S}}V(G)$,
    \item the first and last edge of every maximal path in $F$ are contained in $L$,
    \item the endpoints of every maximal path in $F$ are in $\bigcup_{G\in\mathcal{S}}V_L(G)$, and
  \item the internal vertices of every maximal path in $F$ are disjoint from $\bigcup_{G\in\mathcal{S}}V_L(G)$.
\end{itemize}
Given a connecting-system $(\mathcal{S},F)$, we define its \emph{size} as the number of edges in the intersection $F\cap M$.
Let $H_{(\mathcal S,F)}$ be the multigraph (with loops) on vertex-set $\mathcal S$, where for every maximal path in $F$ with endpoints $u$ and $v$ we add a new (possibly parallel or loop) edge $GG'$ where $G,G'\in\mathcal{S}$ are unique such that $u\in V_L(G)$ and $v\in V_L(G')$. We call a connecting-system \emph{Eulerian} if $H_{(\mathcal S,F)}$ is Eulerian. Given some $G\in \mathcal S$, we denote by $X_{(\mathcal S,F)}(G)\subseteq V(G)$ the set of intervals in $V(G)$ which contain a vertex that appears as an endpoint of a maximal path in $F$. For every $G\in \mathcal{G}\setminus \mathcal{S}$, we define $X_{(\mathcal S, F)}(G):=\emptyset$. 

Next, suppose we are given numbers $p\in [0,1]$ and $k\in\mathbb{N}$, as well as a probability distribution\footnote{In practice, instead of a probability distribution over connecting-systems we will usually deal with a random sample from such a distribution. With a slight abuse of terminology, in the following we will not distinguish between this random object and the probability distribution defining it.} over connecting-systems $(\mathcal S,F)$. We then say that the latter is \emph{$(p,k)$-random} if for every $G\in \mathcal{G}$ the set $X_{(\mathcal S,F)}(G)$, where $(\mathcal S,F)$ is sampled according to said distribution, can be coupled to be included in the union of a $p$-random subset of $V(G)$ and a (deterministic) set of size at most $k$. 

Finally, suppose that $(\mathcal S, F)$ is Eulerian and that $\mathcal C$ is an edge-disjoint cycle decomposition of $H_{(\mathcal S,F)}$. For $G\in \mathcal{S}$ we denote by $R_{(\mathcal S,F)}(G,\mathcal C)$ the graph with vertex-set $V_L(G)$, where we add an edge $uv$ for every cycle $C\in\mathcal C$ going through the vertex $G\in \mathcal S$ of $H_{(\mathcal S,F)}$ such that the two maximal paths of $F$ corresponding to the two edges of $C$ incident to $G$ end in $u$ and $v$, respectively. As the cycles in $\mathcal{C}$ are edge-disjoint and $F$ is a linear forest, $R_{(\mathcal S,F)}(G,\mathcal C)$ is a matching.

This ends our notational and terminological interlude, and we now continue with the proof. To start, let us define $F$ as the linear forest obtained from $L$ by removing all edges $uv$ for which there exists some $G\in\mathcal G$ with $u,v\in V_L(G)$, and removing isolated vertices afterwards. We then have that $F$ is a linear forest, all whose components are paths of positive length whose endpoints are vertices in $\bigcup_{G\in\mathcal{G}}V_L(G)$ at the boundary of some interval in $\bigcup_{G\in \mathcal{G}}V(G)$. Note however that not every vertex at the boundary of an interval in $\bigcup_{G\in \mathcal{G}}V(G)$ must appear as an endpoint of a path in $F$, since there may be consecutive intervals along $L$, both belonging to $V(G)$ for the same graph $G\in\mathcal G$. To avoid $F$ being empty when $|\mathcal G|=1$ and the unique graph $G\in\mathcal G$ has $V(G)=\mathcal I$, in that case we let $F$ be a single, arbitrary edge of $L$ which connects two intervals of $\mathcal I$. Pause to note that with this definition, $(\mathcal G,F)$ is a connecting-system of size $0$. Our goal in the following will be to start from $(\mathcal{G},F)$ and modify this connecting-system step-by-step, each time increasing its size, until we find a connecting-system of large size. We will then extract a cycle in $L\cup M$ using many edges of $M$, which, in particular, will give a path with the same property, as desired.
  \begin{claim}
      $(\mathcal G,F)$ is Eulerian.
  \end{claim}
  \begin{claimproof}
     Let $H$ be the multigraph obtained from $F$ by identifying $V_L(G)$ into a single vertex, for every $G\in\mathcal G$. Then, $H$ is a subdivision of $H_{(\mathcal G,F)}$ (note that since $Z$ may be non-empty, ``subdivision'' here is required).
     Since the edge-set of $F$ was obtained from that of $L$ only by removing edges with both endpoints in the same $G\in\mathcal G$, we find that $H$ is obtained from $L$ by identifying some pairwise disjoint sets of vertices and removing some loops. But note that $L$ is Eulerian and both the operations of identifying disjoint subsets of vertices into single vertices, as well as removing loops, do not change Eulerianness. Thus, we find that $H$ is Eulerian. Furthermore, since the subdivision $H$ of $H_{(\mathcal G,F)}$ is Eulerian, then so is $H_{(\mathcal G,F)}$ itself.
  \end{claimproof}
  As $(\mathcal G,F)$ is Eulerian, we have that $H_{(\mathcal G,F)}$  can be decomposed into a collection of edge-disjoint cycles. Fix such a collection $\mathcal C$ of cycles, and let $w:V(H_{(\mathcal G,F)})=\mathcal{G}\rightarrow \mathbb{Z}_{\ge 0}$ be the weight function defined as $w(G):=|V(G)|$ for $G\in\mathcal G$. Let $\mathcal C^*\subseteq \mathcal C$ be the collection of cycles put forth by applying Lemma~\ref{lem: cycles} to the cycle-collection $\mathcal{C}$ with parameters $k:=n^{\alpha}$, $\delta:=\varepsilon$ and $p:=n^{-\varepsilon\alpha}$. Then, $\Delta(\mathcal C^\ast)\le 6k$. Let $\mathcal C_p\subseteq \mathcal C$ be obtained by selecting every cycle in $\mathcal C$ independently with probability $p$. We denote by $R$ a heaviest connected component of $\mathcal C^*\cup \mathcal C_p$ and note that $R$ is an Eulerian multigraph as it is an edge-disjoint union of cycles. Let $\mathcal R:=V(R)$ be the set of vertices of $R$. Then, by the guarantee of the lemma, we have that with probability at least $(|V(G_M')|-|Z|)^{-\delta}/2>n^{-\varepsilon}$, we have $\sum_{G\in \mathcal{R}}|V(G)|=w(R)\ge (|V(G_M')|-|Z|)^{1-\delta}/2=((2-o(1))n^{1-\varepsilon})^{1-\delta}/2\ge n^{1-2\varepsilon}$, for $n$ large enough. Let $F_0\subseteq F$ be the union of all components of $F$ which correspond to edges in $R$. Then, $(\mathcal R,F_0)$ is a connecting-system of size $0$ such that $H_{(\mathcal R,F_0)}=R$ and, consequently, $(\mathcal R,F_0)$ is Eulerian. Furthermore, $(\mathcal{R},F_0)$ has the following important property when interpreted as a probability distribution over connecting-systems.
  \begin{claim}
      $(\mathcal R,F_0)$ is $(2\sqrt{p},6k)$-random.
  \end{claim}
  \begin{claimproof}
      Note that the paths of $F_0$ can be partitioned into the linear forests $F_{\mathcal C^*}$ and $F_{\mathcal C_p}$ corresponding to the edges of $\mathcal C^*$ and $\mathcal C_p$ respectively. Now consider some arbitrary $G\in \mathcal G$. Then, since $\Delta(\mathcal C^*)\le 6k$, we have that $|X_{(\mathcal R,F_{\mathcal C^*})}(G)|\leq 6k$. Observe that $R_{(\mathcal G,F)}(G,\mathcal C)$ is a matching where the vertices which are not on the boundary of some interval in $V(G)$ have degree $0$. Thus, the graph $R_{(\mathcal G,F)}(G,\mathcal C)\slash \mathcal I$, obtained by identifying intervals into single vertices, is a multigraph of maximum degree at most $2$. Pause to note that since $\mathcal C_p$ contains every cycle of $\mathcal C$ with probability $p$, $X_{(\mathcal R,F_{\mathcal C_p})}(G)$ is a subset of the set of intervals/vertices in $R_{(\mathcal G, F)}(G,\mathcal C)\slash \mathcal I$ which remain incident to at least one edge after subsampling every edge independently with probability $p$. By Lemma~\ref{lem: coupling}, it follows that $X_{(\mathcal R,F_{\mathcal C_p})}(G)$ can be coupled to be included in a $2\sqrt{p}$-random subset of $V(G)$. The assertion of the claim now follows, using that $X_{(\mathcal R, F_0)}(G)=X_{(\mathcal R, F_{\mathcal C^*})}(G)\cup X_{(\mathcal R, F_{\mathcal C_p})}(G)$.
  \end{claimproof}
  In the following, for every $G\in \mathcal G$, let $V_G^\ast\subseteq V(G)$ be a deterministic set with $|V_G^\ast|\le 6k=6n^\alpha$ and $V_G^0\subseteq V(G)$ be a $2\sqrt{p}$-random subset, coupled to $\mathcal{C}_p$ such that in every outcome of the randomness, the inclusion $X_{(\mathcal R, F_0)}(G)\subseteq V_G^\ast \cup V_G^0$ holds. Note that the existence of such a coupling is guaranteed by the previously proved claim. 
  
  Choose $\alpha_0:=2\alpha$ and $\gamma:=\alpha\varepsilon/3$. For sufficiently large $n$, we have $6n^\alpha\le n^{\alpha_0}$ and $2\sqrt p=2n^{-\alpha\varepsilon/2}\le n^{-\gamma}$, so we may couple each $V_G^0$ to be contained in an $n^{-\gamma}$-random superset $\widehat V_G^0\subseteq V(G)$. Now, let $A$ be the event that the probabilistic event described in the statement of Lemma~\ref{lem: expander connections} holds applied to every $G\in\mathcal G$ (with parameters $\alpha_0, \varepsilon,\gamma$ and the associated deterministic and random sets $V_G^\ast, \widehat V_G^0$ taking the role of $V_*$ and $V_0$). Pause to note that Lemma~\ref{lem: expander connections} is indeed applicable here: every $G\in \mathcal G$ is a $(1/8,114,s_G)$-expander with $s_G=\frac{n^\varepsilon}{4(\log |V(G)|)^{114}}$, and our choice of $\alpha$ makes the required inequalities among the constants hold. Combined with a union bound over all $G\in \mathcal{G}$, the lemma guarantees that $A$ happens with probability at least $1-|\mathcal G|\cdot n^{-\varepsilon^2\log n/4}\geq 1-n^{-\varepsilon^2\log n/4+1}$. Also recall that, as established previously, we have $\sum_{G\in \mathcal{R}}|V(G)|\geq n^{1-2\varepsilon}$ with probability at least $n^{-\varepsilon}$. Therefore, for $n$ sufficiently large, with positive probability simultaneously $A$ happens and $\sum_{G\in \mathcal R}|V(G)|\geq n^{1-2\varepsilon}$. In the rest of the proof, let us fix an instance of $\mathcal R$ with these properties. Moving on, we will only use the deterministic properties of this instance and no longer consider any probabilistic events. 

We fix some last piece of notation before the final step of the proof. In the remainder, given a connecting-system $(\mathcal S, T)$, we will use the notation $(\mathcal S,T)\subseteq (\mathcal R,F_0)$ to express that $\mathcal S\subseteq \mathcal R$ and, for every $G\in\mathcal S$, $X_{(\mathcal S,T)}(G)=X_{(\mathcal R,F_0)}(G)$.

  We now prove the following key claim, which allows us to successively modify the connecting-system $(\mathcal{R},F_0)$ while increasing its size, until finally reaching a cycle in $L\cup M$ using a large number of edges in $M$. With this claim at hand, we will then be able to swiftly conclude the proof of the theorem.
  \begin{claim}\label{clm: iterative}
      Let $(\mathcal S,T)\subseteq (\mathcal R,F_0)$ be an Eulerian connecting-system of size $s$. Let $G\in\mathcal S$. Then, either there exists $T'\supseteq T$ such that $(\mathcal S\setminus \{G\},T')\subseteq (\mathcal R,F_0)$ is an Eulerian connecting-system of size at least $s+n^{-\varepsilon}\cdot |V(G)|$ or $\mathcal S=\{G\}$ and there exists a cycle in $L\cup M$ using at least $s+n^{-\varepsilon}\cdot |V(G)|$ edges of $M$.
  \end{claim}
  \begin{claimproof}
    Let $\mathcal C$ be an edge-disjoint cycle decomposition of $H_{(\mathcal S,T)}$. Recall that $X_{(\mathcal S,T)}(G)$ is, by definition, the set of intervals with positive degree in $R_{(\mathcal S,T)}(G,\mathcal C)\slash\mathcal I$. Also, note that, since $(\mathcal S, T)$ is an Eulerian connecting-system, $H_{(\mathcal S, T)}$ is an Eulerian graph and hence there is at least one cycle in $\mathcal C$ passing through $G$. It follows then by definition that $X_{(\mathcal S,T)}(G)\neq \emptyset$ and that the graph $R_{(\mathcal S,T)}(G,\mathcal C)\slash\mathcal I$ has at least one edge. Let us now define $R:=(R_{(\mathcal S,T)}(G,\mathcal C)\slash\mathcal I)[X_{(\mathcal S,T)}(G)]$ as the multigraph obtained from $R_{(\mathcal S,T)}(G,\mathcal C)\slash\mathcal I$ by removing all isolated vertices. Note that $R$ has maximum degree at most two and contains at least one edge since $T$ is non-empty. Recall that we assumed the event $A$ holds. Thus, we can now apply the conclusion of Lemma~\ref{lem: expander connections} to the expander $G$, and the multigraph $R$ with vertex-set $W:=X_{(\mathcal S,T)}(G)\subseteq V_G^\ast\cup \hat V_G^0$. The lemma then yields a collection $\mathcal P$ of internally vertex-disjoint arcs in $G$ with the following properties.
    \begin{itemize}
    \item Every arc $P\in \mathcal P$ has all its internal vertices in $V(G)\setminus W$ and its endpoints in $W$. 
        \item Every cycle $C$ in $R$ has a vertex $v_C$ such that no vertex in $V(C)\setminus \{v_C\}$ is an endpoint of any arc in $\mathcal P$, and the graph $R'$ obtained from $R$ by removing all edges and all vertices of every cycle $C$, except $v_C$, and adding an edge $uv$ for every arc $P\in \mathcal{P}$ with endpoints $u$ and $v$, is a cycle. The latter in particular implies that for every cycle $C$ in $R$, the vertex $v_C$ is the endpoint of exactly \emph{two} arcs in $\mathcal P$ (or one closed arc in $\mathcal P$, in case that $\mathcal P$ only consists of a single arc).
        \item $\mathcal P$ contains a total of at least $|V(G)|\cdot n^{-\gamma}/2\geq n^{-\varepsilon}|V(G)|$ edges.
    \end{itemize}
    Observe that the edges on the arcs in $\mathcal{P}$ are edges in $G\subseteq G_M'$, and hence by definition of $G_M'$ they correspond to edges in the matching $M$. In the following, let us denote by $M_{\mathcal P}\subseteq M$ the submatching of $M$ containing all matching edges corresponding to edges on some arc in $\mathcal P$. By the above, we then have $|M_\mathcal P|\ge n^{-\varepsilon} |V(G)|$.

    Suppose $I\in V(G)$ is an interval which appears as an internal vertex of some arc $P\in\mathcal P$. Then, exactly two vertices of $I$ are incident to edges in $M_{\mathcal P}$, namely to the two edges in $M_{\mathcal P}$ corresponding to the two edges of $P$ incident to $I$. Let $I'\subseteq I$ be the subinterval connecting these two vertices. Let $F\subseteq L\cup M$ be the subgraph obtained as the union of the matching $M_{\mathcal P}$ with the subpath of $L$ induced by the vertices in $I'$ for every interval $I$ which is an internal vertex of one of the arcs in $\mathcal P$. Note that since the arcs in $\mathcal P$ are internally vertex-disjoint, the graph $F$ defined in this way is a linear forest. Note further that the components of $F$ correspond to the arcs in $\mathcal P$, using the same edges of $M$ and intersecting the same intervals, so $F$ also uses at least $n^{-\varepsilon}|V(G)|$ edges of $M$. Let $V\subseteq V(R)$ be the set of intervals which are endpoints of arcs in $\mathcal P$. Note that $V$ can be partitioned into $V_1$, the intervals with degree $1$ in $R$, and $V_2$, the intervals of the form $v_C$ for some cycle $C$ in $R$. As every interval $I$ in $V_1$ has degree $1$ in $R$, there exists exactly one maximal path in $T$ which ends in a vertex $v\in I$. Since every path in $T$ ends in an edge of $L$ while not containing edges of $L$ inside $I$, $v$ must be a boundary vertex of $I$. Let us extend all the paths in $F$ which end in an interval in $V_1$ to the corresponding vertex $v$ at the boundary via the unique connecting subpath of $L$ within that interval. For every path of $F$ which ends in an interval $I\in V_2$, let us also extend this path to a vertex at the boundary of the interval in such a way that it does not intersect the other vertex on $I$ which is an endpoint of a path in $F$. Since every interval in $V_2$ is the endpoint of exactly $2$ paths in $F$, there is a unique way to do this. Let $F'\supseteq F$ be the resulting graph. Finally, let $F''\supseteq F'$ be obtained by adding, for every interval $I\in V(R)\setminus V$, the whole subpath of $L$ induced by $I$ as a path to $F''$. 

    Next, let us observe the following property of $F''$. For every path in $T$ with endpoint $u\in V_L(G)$, there is exactly one path in $F''$ which also ends in $u$. Besides these intersections, the paths in $F''$ and $T$ are disjoint, since all the vertices of $F''$ are in $V_L(G)$ while only the endpoints of $T$ may be in $V_L(G)$. Therefore, $T':=F''\cup T$ is a graph of maximum degree $2$, where every vertex of degree $1$ in $T$ which is not in $V_L(G)$ still has degree $1$ and no vertex in $V_L(G)$ has degree $1$. 

    Consider the graph obtained from $H_{(\mathcal S,T)}$ by replacing the vertex $G$ by the new set of vertices $V(F'')$; replacing every non-loop edge $GG'$ incident to $G$ by the edge $vG'$, where $v\in V(F'')$ is the unique endpoint in $V_L(G)$ of the path in $T$ corresponding to $GG'$; and replacing every loop at $G$ by an edge $uv$, where $u,v\in V(F'')$ are the two endpoints in $V_L(G)$ of the corresponding path in $T$. Let $H$ be the union of this graph and $F''$. We claim that $H$ is a connected graph. Since $H_{(\mathcal S,T)}$ is connected, every connected component of $H$ contains a vertex of $V(F'')$. Recall that every edge in $R$ corresponds to a cycle $C\in \mathcal C$ through $G$. Furthermore, if $u,v$ are the vertices in $G$ which appear on the paths of $T$ corresponding to the edges incident to $G$ in $C$, then we made sure that $u,v\in V(F'')$. For each such pair $uv$, we have that they are connected in $H$ through the cycle $C$. Let $H'$ be obtained from $H$ by adding the edge $uv$ for every such pair $u,v$ obtained from an edge of $R$, and note that $H'$ is connected if and only if $H$ is. Next, consider $H'':=H'[V(F'')]$ and note that if $H''$ is connected then so is $H$ since, as mentioned earlier, every connected component of $H$ (and hence of $H'$) intersects $V(F'')$. But by the properties of $\mathcal P$ guaranteed by Lemma~\ref{lem: expander connections} and the derivation of $F''$ from $\mathcal P$, $H''$ is a spanning cycle and, thus, connected.

    Suppose for a moment that $G$ is the only element of $\mathcal S$. Then, when we added the edge $uv$ when defining $H'$ in the above argument, this was actually already an edge of $H$ and similarly, we did not remove any vertices since $H$ has no other vertices besides $V(F'')$. So, $H=H'=H''$ is a spanning cycle in this case. Replacing every edge of $H$ which corresponds to a path in $T$ by this path we obtain a cycle in $L\cup M$ incorporating all the paths of $T$ as well as all the edges of $M_{\mathcal P}$. Therefore, this cycle contains at least $s+n^{-\varepsilon}|V(G)|$ edges of $M$, concluding the proof of the claim in this case.

    So, moving on, let us assume that $G$ is not the only element of $\mathcal S$. In the rest of the argument, we aim to show that $T'=F''\cup T$ is the desired linear forest, satisfying the properties required by the claim we set out to prove. 

First we prove that $T'$ is a linear forest. We already noted above that $T'$ has maximum degree at most two, so it remains to be shown that $T'$ contains no cycles. Suppose towards a contradiction that $T'$ contains a cycle. Then, this cycle must consist of the union of a non-empty set of maximal paths in $F''$ and a non-empty set of maximal paths in $T$. Moreover, all endpoints of the involved paths must be contained in $V_L(G)$, as any endpoint $v\notin V_L(G)$ of a maximal path in $F''$ or $T$ has degree $1$ in $T'$ and hence cannot be contained in a cycle.

It now follows that the cycle in $T'$ we considered also constitutes a cycle in $H''$, which we have shown is a spanning cycle, and so the cycle we considered must equal that spanning cycle. Since further the maximum degree of $T'$ is at most two, it then follows that every maximal path in $T$ has either no or exactly two endpoints in $V_L(G)$. It now follows that $G$ is an isolated vertex in $H_{(\mathcal S,T)}$ contradicting that we are assuming $\mathcal S\setminus \{G\}\neq \emptyset$ and that $H_{(\mathcal S,T)}$ is Eulerian and hence connected. This contradiction shows that our above assumption was false, indeed, $T'$ does not contain any cycles and, hence, is a linear forest. 

    Recall that the endpoints of maximal paths in $T'$ are exactly the endpoints of paths in $T$ which are not in $V_L(G)$. Furthermore, we did not change the edges incident to these endpoints in $T'$, so they are still edges of the cycle $L$. The internal vertices of the paths in $T'$ are either internal vertices of paths in $T$ or from $V_L(G)$. Therefore, the internal vertices of paths in $T'$ are disjoint from $\bigcup_{G'\in \mathcal S\setminus\{G\}}V_L(G')$. Let us set $\mathcal S':=\mathcal S\setminus\{G\}$. We claim that now, $(\mathcal S',T')$ is a connecting-system: Pause to note that the graph $H$ defined above is a subdivision of $H_{(\mathcal S',T')}$. Thus, since we showed that $H$ is connected, it follows that $H_{(\mathcal S',T')}$ is connected, too. Also, the degree of every vertex in $H_{(\mathcal S',T')}$ is the same as in $H_{(\mathcal S,T)}$, so $(\mathcal S',T')$ is also Eulerian. Furthermore, $X_{(\mathcal S',T')}(G')=X_{(\mathcal S,T)}(G')$ for every $G'\in\mathcal S\setminus \{G\}$ so that $(\mathcal S',T')\subseteq (\mathcal R,F_0)$. Finally, we have that $|T'\cap M|=|T\cap M|+|M_{\mathcal P}|\geq s+n^{-\varepsilon}|V(G)|$. This is one of the two desired outcomes of the claim, and this concludes the proof of the latter.
  \end{claimproof} 
  Let us apply Claim~\ref{clm: iterative} iteratively, starting from the connecting-system $(\mathcal R,F_0)$, until we obtain a cycle in $L\cup M$. Note that this happens exactly after $|\mathcal R|$ applications of the claim. Tracking the size of the connecting-system through this process, we conclude that the resulting cycle contains at least $\sum_{G\in\mathcal R}n^{-\varepsilon}|V(G)|=n^{-\varepsilon}\sum_{G\in \mathcal R}|V(G)|\ge n^{1-3\varepsilon}$ edges of $M$. In particular, there exists a path in $L\cup M$ using at least $n^{1-3\varepsilon}-1$ edges of $M$. This is what we initially set out to prove, and concludes the proof of the theorem.
\end{proof}
\paragraph{Statement of AI use:} ChatGPT 5.6 was used to proofread this paper. All of the ideas and writing are due entirely to the authors.
\bibliographystyle{abbrv}
\bibliography{sources.bib}

@article {expanders,
    AUTHOR = {Letzter, Shoham and Methuku, Abhishek and Sudakov, Benny},
     TITLE = {Nearly {H}amilton cycles in sublinear expanders and
              applications},
   JOURNAL = {J. Lond. Math. Soc. (2)},
  FJOURNAL = {Journal of the London Mathematical Society. Second Series},
    VOLUME = {113},
      YEAR = {2026},
    NUMBER = {2},
     PAGES = {Paper No. e70452, 42},
      ISSN = {0024-6107,1469-7750},
   MRCLASS = {05C45 (05C35 05C48 05C83)},
  MRNUMBER = {5032574},
       DOI = {10.1112/jlms.70452},
       URL = {https://doi.org/10.1112/jlms.70452},
}

@article {groen,
    AUTHOR = {Groenland, Carla and Longbrake, Sean and Steiner, Raphael and
              Turcotte, J\'er\'emie and Yepremyan, Liana},
     TITLE = {Longest cycles in vertex-transitive and highly connected
              graphs},
   JOURNAL = {Bull. Lond. Math. Soc.},
  FJOURNAL = {Bulletin of the London Mathematical Society},
    VOLUME = {57},
      YEAR = {2025},
    NUMBER = {10},
     PAGES = {2975--2990},
      ISSN = {0024-6093,1469-2120},
   MRCLASS = {05C38 (05C45 05C60 68V05)},
  MRNUMBER = {4971882},
       DOI = {10.1112/blms.70134},
       URL = {https://doi.org/10.1112/blms.70134},
}

@article {norin,
    AUTHOR = {Norin, Sergey and Steiner, Raphael and Thomass\'e, St\'ephan
              and Wollan, Paul},
     TITLE = {Small hitting sets for longest paths and cycles},
   JOURNAL = {Proc. Amer. Math. Soc.},
  FJOURNAL = {Proceedings of the American Mathematical Society},
    VOLUME = {154},
      YEAR = {2026},
    NUMBER = {6},
     PAGES = {2319--2335},
      ISSN = {0002-9939,1088-6826},
   MRCLASS = {05C38 (05C69)},
  MRNUMBER = {5065039},
       DOI = {10.1090/proc/17582},
       URL = {https://doi.org/10.1090/proc/17582},
}

@article{devos,
  title={Longer cycles in vertex transitive graphs},
  author={DeVos, Matt},
  journal={preprint arXiv:2302.04255},
doi= {10.48550/arXiv.2302.04255},
  year={2023}
}

@article{ma2025intersections,
  title={Intersections of longest cycles in vertex-transitive and highly connected graphs},
  author={Ma, Jie and Zhao, Ziyuan},
  journal={preprint arXiv:2508.17438},
  year={2025}
}

@incollection {karp,
    AUTHOR = {Karp, Richard M.},
     TITLE = {Reducibility among combinatorial problems},
 BOOKTITLE = {Complexity of computer computations},
    SERIES = {The IBM Research Symposia},
     PAGES = {85--103},
      YEAR = {1972},
   MRCLASS = {68A20},
  MRNUMBER = {378476},
MRREVIEWER = {John\ T.\ Gill},
doi = {10.1007/978-1-4684-2001-2_9}
}

@inproceedings{lovasz1969,
  title={Problem 11},
  author={Lov{\'a}sz, L{\'a}szl{\'o}},
  booktitle={Combinatorial Structures and Their Applications, Proc. Calgary Internat. Conf. on Combinatorial structures and their applications, 1969},
  publisher={Gordon and Breach},
  pages={497},
  year={1970},
}

@article {babai,
    AUTHOR = {Babai, L\'aszl\'o},
     TITLE = {Long cycles in vertex-transitive graphs},
   JOURNAL = {J. Graph Theory},
  FJOURNAL = {Journal of Graph Theory},
    VOLUME = {3},
      YEAR = {1979},
    NUMBER = {3},
     PAGES = {301--304},
   MRCLASS = {05C35 (05C38)},
  MRNUMBER = {542553},
       DOI = {10.1002/jgt.3190030314},
}

@article {cayley-survey-84,
    AUTHOR = {Witte, David and Gallian, Joseph A.},
     TITLE = {A survey: {H}amiltonian cycles in {C}ayley graphs},
   JOURNAL = {Discrete Math.},
  FJOURNAL = {Discrete Mathematics},
    VOLUME = {51},
      YEAR = {1984},
    NUMBER = {3},
     PAGES = {293--304},
   MRCLASS = {05C45 (05C25)},
  MRNUMBER = {762322},
MRREVIEWER = {A.\ Gregory\ Starling},
       DOI = {10.1016/0012-365X(84)90010-4},
}

@incollection {alspach81,
    AUTHOR = {Alspach, Brian},
     TITLE = {The search for long paths and cycles in vertex-transitive
              graphs and digraphs},
 BOOKTITLE = {Combinatorial mathematics, {VIII} ({G}eelong, 1980)},
    SERIES = {Lecture Notes in Math.},
    VOLUME = {884},
     PAGES = {14--22},
 PUBLISHER = {Springer, Berlin},
      YEAR = {1981},
      ISBN = {3-540-10883-1},
   MRCLASS = {05C38},
  MRNUMBER = {641232},
}

@article {kutnar,
    AUTHOR = {Kutnar, Klavdija and Maru{\v si\v c}, Dragan},
     TITLE = {Hamilton cycles and paths in vertex-transitive
              graphs---current directions},
   JOURNAL = {Discrete Math.},
  FJOURNAL = {Discrete Mathematics},
    VOLUME = {309},
      YEAR = {2009},
    NUMBER = {17},
     PAGES = {5491--5500},
      ISSN = {0012-365X,1872-681X},
   MRCLASS = {05-02 (05C25)},
  MRNUMBER = {2548567},
       DOI = {10.1016/j.disc.2009.02.017},
       URL = {https://doi.org/10.1016/j.disc.2009.02.017},
}

@article{bedert,
  title={The {L}ov{\'a}sz conjecture holds for moderately dense {C}ayley graphs},
  author={Bedert, Benjamin and Dragani{\'c}, Nemanja and M{\"u}yesser, Alp and Pavez-Sign{\'e}, Mat{\'\i}as},
  journal={preprint arXiv:2603.08675},
  year={2026}
}

@article{bradac,
  title={Hamiltonicity of regular sublinear expanders},
  author={Brada{\v{c}}, Domagoj and Janzer, Oliver},
  journal={preprint arXiv:2605.15043},
  year={2026}
}

@article {RS,
    AUTHOR = {Rapaport-Strasser, Elvira },
     TITLE = {Cayley color groups and {H}amilton lines},
   JOURNAL = {Scripta Math.},
  FJOURNAL = {Scripta Mathematica},
    VOLUME = {24},
      YEAR = {1959},
     PAGES = {51--58},
      ISSN = {0036-9713},
   MRCLASS = {05.00 (20.00)},
  MRNUMBER = {111702},
MRREVIEWER = {M.\ Fiedler},
}

@article {rankin,
    AUTHOR = {Rankin, R. A.},
     TITLE = {A campanological problem in group theory},
   JOURNAL = {Proc. Cambridge Philos. Soc.},
  FJOURNAL = {Proceedings of the Cambridge Philosophical Society},
    VOLUME = {44},
      YEAR = {1948},
     PAGES = {17--25},
      ISSN = {0008-1981},
   MRCLASS = {20.0X},
  MRNUMBER = {22846},
MRREVIEWER = {H.\ S. M. Coxeter},
       DOI = {10.1017/s030500410002394x},
       URL = {https://doi-org.ezproxy.princeton.edu/10.1017/s030500410002394x},
}

@article{bucic,
  title={Long cycles in vertex transitive digraphs},
  author={Buci{\'c}, Matija and Hendrey, Kevin and Mohar, Bojan and Steiner, Raphael and Yepremyan, Liana},
  journal={preprint arXiv:2602.16333},
  year={2026}
}

@article{draganic,
  title={Hamiltonicity of expanders: optimal bounds and applications},
  author={Dragani{\'c}, Nemanja and Montgomery, Richard and Correia, David Munh{\'a} and Pokrovskiy, Alexey and Sudakov, Benny},
  journal={preprint arXiv:2402.06603},
  year={2024}
}

@article {watkins,
    AUTHOR = {Watkins, Mark E.},
     TITLE = {Connectivity of transitive graphs},
   JOURNAL = {J. Combinatorial Theory},
  FJOURNAL = {Journal of Combinatorial Theory},
    VOLUME = {8},
      YEAR = {1970},
     PAGES = {23--29},
      ISSN = {0021-9800},
   MRCLASS = {05.40},
  MRNUMBER = {266804},
MRREVIEWER = {E.\ A.\ Nordhaus},
}

@article {bondy,
    AUTHOR = {Bondy, J. A. and Locke, S. C.},
     TITLE = {Relative lengths of paths and cycles in {$3$}-connected
              graphs},
   JOURNAL = {Discrete Math.},
  FJOURNAL = {Discrete Mathematics},
    VOLUME = {33},
      YEAR = {1981},
    NUMBER = {2},
     PAGES = {111--122},
      ISSN = {0012-365X,1872-681X},
   MRCLASS = {05C38},
  MRNUMBER = {599075},
MRREVIEWER = {Carsten\ Thomassen},
       DOI = {10.1016/0012-365X(81)90159-X},
       URL = {https://doi.org/10.1016/0012-365X(81)90159-X},
}

@InProceedings{gallai,
  author =	{Gallai, Tibor},
  title =	{{Problem 4}},
  booktitle =	{Proceedings
of the Colloquium Held at Tihany, Hungary, September 1966},
  pages =	{362},
  series =	{},
 
  year =	{1968},
  volume =	{},
  editor =	{Erd\H{o}s, Paul and Katona, Gyulia},
  publisher =	{Academic Press},
  address =	{New York} 
}

@article {walther,
    AUTHOR = {Walther, Hansjoachim},
     TITLE = {{\"U}ber die {N}ichtexistenz eines {K}notenpunktes, durch den
              alle l\"angsten {W}ege eines {G}raphen gehen},
   JOURNAL = {J. Combinatorial Theory},
  FJOURNAL = {Journal of Combinatorial Theory},
    VOLUME = {6},
      YEAR = {1969},
     PAGES = {1--6},
      ISSN = {0021-9800},
   MRCLASS = {05.40},
  MRNUMBER = {236054},
MRREVIEWER = {W.\ Moser},
}

@book {Wal78,
    AUTHOR = {Walther, Hansjoachim},
     TITLE = {Anwendungen der {G}raphentheorie},
 PUBLISHER = {VEB Deutscher Verlag der Wissenschaften, Berlin},
      YEAR = {1979},
     PAGES = {239},
   MRCLASS = {90B10 (05Cxx 90C35)},
  MRNUMBER = {539852},
}

@article {Zam76,
    AUTHOR = {Zamfirescu, Tudor},
     TITLE = {On longest paths and circuits in graphs},
   JOURNAL = {Math. Scand.},
  FJOURNAL = {Mathematica Scandinavica},
    VOLUME = {38},
      YEAR = {1976},
    NUMBER = {2},
     PAGES = {211--239},
      ISSN = {0025-5521,1903-1807},
   MRCLASS = {05C35},
  MRNUMBER = {429645},
MRREVIEWER = {Carsten\ Thomassen},
       DOI = {10.7146/math.scand.a-11630},
       URL = {https://doi.org/10.7146/math.scand.a-11630},
}

@inproceedings{Groetschel84,
  author    = {Martin Gr{\"o}tschel},
  title     = {On intersections of longest cycles},
  booktitle = {Graph theory and combinatorics: proceedings of the Cambridge Combinatorial Conference in honour of Paul Erd{\"o}s},
  editor    = {B{\'e}la Bollob{\'a}s},
  pages     = {171 -- 189},
  year      = {1984},
url={https://www.researchgate.net/publication/266281681_On_intersections_of_longest_cycles}
}

@incollection{Bondy95,
    author = {J. A. Bondy},
    title = {Basic graph theory: paths and circuits},
    booktitle = {Handbook of Combinatorics},
    publisher = {Elsevier},
    year = {1995},
url={https://dl.acm.org/doi/10.5555/233157.233163}
}

@article {Lon21,
    AUTHOR = {Long, Jr., James A. and Milans, Kevin G. and Munaro, Andrea},
     TITLE = {Sublinear longest path transversals},
   JOURNAL = {SIAM J. Discrete Math.},
  FJOURNAL = {SIAM Journal on Discrete Mathematics},
    VOLUME = {35},
      YEAR = {2021},
    NUMBER = {3},
     PAGES = {1673--1677},
      ISSN = {0895-4801,1095-7146},
   MRCLASS = {05C38 (05C40 05C69 05D15)},
  MRNUMBER = {4291372},
MRREVIEWER = {Vahan\ V.\ Mkrtchyan},
       DOI = {10.1137/20M1362577},
       URL = {https://doi.org/10.1137/20M1362577},
}

@article {Kie23,
    AUTHOR = {Kierstead, H. A. and Ren, E. R.},
     TITLE = {Improved upper bounds on longest-path and maximal-subdivision
              transversals},
   JOURNAL = {Discrete Math.},
  FJOURNAL = {Discrete Mathematics},
    VOLUME = {346},
      YEAR = {2023},
    NUMBER = {9},
     PAGES = {Paper No. 113514, 5},
      ISSN = {0012-365X,1872-681X},
   MRCLASS = {05C38 (05C12 05C69)},
  MRNUMBER = {4592301},
MRREVIEWER = {Vahan\ V.\ Mkrtchyan},
       DOI = {10.1016/j.disc.2023.113514},
       URL = {https://doi.org/10.1016/j.disc.2023.113514},
}

@article {Rau14,
    AUTHOR = {Rautenbach, Dieter and Sereni, Jean-S\'ebastien},
     TITLE = {Transversals of longest paths and cycles},
   JOURNAL = {SIAM J. Discrete Math.},
  FJOURNAL = {SIAM Journal on Discrete Mathematics},
    VOLUME = {28},
      YEAR = {2014},
    NUMBER = {1},
     PAGES = {335--341},
      ISSN = {0895-4801,1095-7146},
   MRCLASS = {05C38 (05C70)},
  MRNUMBER = {3174168},
MRREVIEWER = {Haruhide\ Matsuda},
       DOI = {10.1137/130910658},
       URL = {https://doi.org/10.1137/130910658},
}

@article {Chen,
    AUTHOR = {Chen, Guantao and Faudree, Ralph J. and Gould, Ronald J.},
     TITLE = {Intersections of longest cycles in {$k$}-connected graphs},
   JOURNAL = {J. Combin. Theory Ser. B},
  FJOURNAL = {Journal of Combinatorial Theory. Series B},
    VOLUME = {72},
      YEAR = {1998},
    NUMBER = {1},
     PAGES = {143--149},
      ISSN = {0095-8956,1096-0902},
   MRCLASS = {05C38 (05C40)},
  MRNUMBER = {1604717},
MRREVIEWER = {Andreas\ Huck},
       DOI = {10.1006/jctb.1997.1802},
       URL = {https://doi.org/10.1006/jctb.1997.1802},
}

@article {torsten,
    AUTHOR = {Merino, Arturo and M\"utze, Torsten and Namrata},
     TITLE = {Kneser graphs are {H}amiltonian},
   JOURNAL = {Adv. Math.},
  FJOURNAL = {Advances in Mathematics},
    VOLUME = {468},
      YEAR = {2025},
     PAGES = {Paper No. 110189, 80},
      ISSN = {0001-8708,1090-2082},
   MRCLASS = {05C45},
  MRNUMBER = {4876460},
MRREVIEWER = {Peter\ Hor\'ak},
       DOI = {10.1016/j.aim.2025.110189},
       URL = {https://doi.org/10.1016/j.aim.2025.110189},
}

\appendix

\section{Proof of Theorem~\ref{thm:smallhittingset}, assuming Theorem~\ref{thm: main}}
As announced in the introduction, in this section, we give the proof of Theorem~\ref{thm:smallhittingset} assuming Theorem~\ref{thm: main}. As mentioned, the proof of this implication is fully analogous to the proofs of Norin et al.~\cite{norin} and hence is included here only for the sake of completeness, and no claim of originality is made. To allow for a short presentation, we use the following, mostly well-known, auxiliary results explicitly proved by Norin et al.

\begin{lemma}[\cite{norin}, Lemma~4]\label{lem:a}
Let $G$ be a graph, let $L$ be a longest path in $G$, and let $C$ be a cycle in $G$ such that $V(L)\cap V(C)=\emptyset$. Then, there exists a set of at most $\sqrt{2|L|}$ vertices in $G$ which separates $V(L)$ and $V(C)$. 
\end{lemma}
\begin{lemma}[\cite{norin}, Lemma~5(i)]\label{lem:b}
Let $G$ be a connected graph. Then, $\mathrm{lpt}(G)=1$ or $G$ contains a cycle intersecting all longest paths in $G$.
\end{lemma}
\begin{lemma}[\cite{norin}, Lemma~6]\label{lem:disjointpairs}
Let $C$ be a cycle, let $A$ and $B$ be disjoint segments of $C$, and let $a_1,\ldots,a_t\in A$, $b_1,\ldots,b_t\in B$ be distinct. Then, $$\sum_{i=1}^{t}\mathrm{dist}_C(a_i,b_i)\ge \frac{t^2}{2}.$$   
\end{lemma}
To state the fourth result from~\cite{norin} that we will use, we need to introduce some additional terminology. 

A \emph{society} consists of a pair $(G,\Omega)$, where $G$ is a graph and $\Omega$ denotes a cyclic order on a subset of vertices of $G$. We use $V(\Omega)$ to denote the subset of vertices appearing in $\Omega$. A \emph{segment} of a society $(G,\Omega)$ is a subset of $V(\Omega)$ which appears contiguously in the cyclic ordering $\Omega$. An \emph{$\Omega$-path} is a path in $G$ with distinct endpoints both in $V(\Omega)$ and all internal vertices outside of $V(\Omega)$. A \emph{transaction} of order $k$ in $(G,\Omega)$ is a set of $k$ pairwise disjoint $\Omega$-paths for which there exist two disjoint segments $A, B$ of $\Omega$ such that each of the $k$ paths has one endpoint in $A$ and the other in $B$. 

The following lemma was implicitly proved by Norin et al., which can be verified by inspecting the proof of Claim~1 within the proof of Theorem~1 of the paper~\cite{norin}. 

\begin{lemma}\label{lem:c}
Let $G$ be a connected graph, and let $C$ be a cycle in $G$ intersecting all longest paths. Let $\Omega$ denote the society in $G$ corresponding to the natural cyclic order on the cycle $C$ in $G$. Let $k\in \mathbb{N}$ be the largest integer such that $(G,\Omega)$ admits a transaction of order $k$. Then, we have 
$$\mathrm{lpt}(G)\le 2k+1.$$
\end{lemma}

With these auxiliary statements at hand, we are now ready to present the proof of Theorem~\ref{thm:smallhittingset} assuming Theorem~\ref{thm: main}.

\begin{proof}[Proof of Theorem~\ref{thm:smallhittingset}]
We have to prove that for every fixed $\varepsilon>0$ and every sufficiently large integer $\ell$, every connected graph $G$ on at least two vertices with maximum path length $\ell$ satisfies $\mathrm{lpt}(G)\le \ell^{1/2+\varepsilon}$.  By Lemma~\ref{lem:b}, we have $\mathrm{lpt}(G)=1$, or there exists a cycle in $G$ which intersects all longest paths. Since we clearly have $\mathrm{lpt}(G)\le \ell^{1/2+\varepsilon}$ in the first case, moving on, we may assume that there is a cycle intersecting all longest paths in $G$, and that $C$ is chosen among all such cycles as short as possible. We now distinguish two possible cases.

\textbf{Case~1.} $C$ is not a geodetic cycle (that is, there exists a pair of vertices $x,y$ on $C$ such that $\mathrm{dist}_G(x,y)<\mathrm{dist}_C(x,y)$). It is then not hard to see that there must exist a path $P$ in $G$, with distinct endpoints $a,b\in V(C)$, such that $P$ is internally vertex-disjoint from $C$, and such that $P$ is strictly shorter than the shortest path between $a$ and $b$ on $C$. Let $C_1, C_2$ be the two cycles in $G$ obtained from the union of $P$ and the two segments of $C$ with endpoints $a,b$. Then, we have that both $C_1, C_2$ are shorter than $C$, and hence by our choice of $C$ there must exist longest paths $L_1, L_2$ in $G$ such that $V(L_i)\cap V(C_i)=\emptyset$ for $i=1,2$. By Lemma~\ref{lem:a} there then exist sets $S_1, S_2\subseteq V(G)$, with $|S_1|, |S_2|\le \sqrt{2\ell}$ such that $S_i$ separates $V(L_i)$ and $V(C_i)$ in $G$, for both $i\in \{1,2\}$. Now consider an arbitrary longest path $L$ in $G$. By the choice of $C$, we have that $L$ intersects $C$ and hence $L$ intersects $C_i$ for some $i\in \{1,2\}$. Furthermore, since any two longest paths in a connected graph intersect, we also have that $L$ intersects $V(L_i)$. But then it also must contain a vertex of $S_i$. It follows that $S_1\cup S_2$ is a longest path transversal in $G$, and hence that $\mathrm{lpt}(G)\le |S_1\cup S_2|\le 2\sqrt{2\ell}\le \ell^{1/2+\varepsilon}$, using our assumption that $\ell$ was chosen sufficiently large in the last inequality. This concludes the proof in this first case.

\textbf{Case~2.} $C$ is a geodetic cycle in $G$, i.e., for any two vertices $x,y\in V(C)$, we have $\mathrm{dist}_G(x,y)=\mathrm{dist}_C(x,y)$. Let now $k$ denote the largest integer such that the society $(G,\Omega)$ defined by $V(C)$ equipped with the natural cyclic ordering along the cycle $C$ admits a transaction of order $k$. Then, by Lemma~\ref{lem:c} we have $\mathrm{lpt}(G)\le 2k+1$, and so it remains to upper-bound $k$. 

Let $\mathcal{P}=\{P_1,\ldots,P_k\}$ be a transaction of $(G,\Omega)$ of order $k$ and let $A, B$ be disjoint segments of $\Omega$ such that each path $P_i\in \mathcal{P}$ has one endpoint $a_i$ in $A$ and the other endpoint $b_i$ in $B$.

Let us now define an auxiliary graph $H$ on vertex-set $\{a_1,\ldots,a_k,b_1,\ldots,b_k\}$ whose edges include all pairs $xy$ with $x,y\in V(H)$ such that $x,y$ are consecutive elements of $V(H)$ along the cycle $C$ in $G$, as well as all pairs $a_ib_i, i=1,\ldots,k$. Then, $H$ satisfies the preconditions of Theorem~\ref{thm: main}, which implies that there exists a path $Q$ in the graph $H$ such that $Q$ uses at least $k^{1-o_k(1)}$ edges of the form $a_ib_i$. Now, let $R$ be the path in $G$ obtained from $Q$ by expanding all edges $xy\in E(Q)$ corresponding to consecutive elements of $V(H)$ on the cycle $C$ into the corresponding segment of $C$ in $G$, and replacing all edges of the form $a_ib_i\in E(Q)$ by the path $P_i$. Let $I\subseteq [k]$ denote the set of all indices $i$ such that $a_ib_i\in E(Q)$ (and thus $P_i\subseteq R$). Note that since $C$ is a geodetic cycle of $G$, we have $|P_i|\ge \mathrm{dist}_C(a_i,b_i)$ for every $i\in [k]$. Since the paths $P_i, i\in I$ are pairwise disjoint, we now have
$$\ell\ge |R|\ge \sum_{i\in I}|P_i|\ge \sum_{i\in I}\mathrm{dist}_C(a_i,b_i).$$
It now follows from~Lemma~\ref{lem:disjointpairs} that the last expression in the above inequality chain must be at least $\frac{|I|^2}{2}$, and thus we obtain $\ell\ge \frac{|I|^2}{2}$, and so after rearranging, we obtain that $k^{1-o_k(1)}\le |I|\le \sqrt{2\ell}$. This implies that $k\le \ell^{1/2+o_\ell(1)}$, and hence that $\mathrm{lpt}(G)\le 2k+1\le \ell^{1/2+o_\ell(1)}\le \ell^{1/2+\varepsilon}$ for every sufficiently large $\ell$. This is what we wanted to prove. Having obtained the desired outcome in both Case~1 and~2, we may now conclude the proof.
\end{proof}
\end{document}